\documentclass{amsart}

\usepackage{amsmath,amssymb,amscd}
\usepackage[left=25mm,right=25mm,top=25mm,bottom=25mm]{geometry}
\usepackage{enumerate}
\usepackage{graphicx}
\usepackage[
	pdftitle={Orbifold Hurwitz numbers and Eynard-Orantin invariants},
	pdfauthor={Norman Do, Oliver Leigh, and Paul Norbury},
	ocgcolorlinks,
	linkcolor=linkblue,
	citecolor=linkred,
	urlcolor=linkred]
{hyperref}
\usepackage{microtype}
\usepackage{multicol}
\usepackage{booktabs}
\usepackage{multirow}
\usepackage{setspace}
\usepackage{color}
	\definecolor{linkred}{rgb}{0.6,0,0}
	\definecolor{linkblue}{rgb}{0,0,0.6}

\theoremstyle{plain}
\newtheorem{theorem}{Theorem}
\newtheorem{lemma}[theorem]{Lemma}
\newtheorem{proposition}[theorem]{Proposition}

\newtheorem{corollary}[theorem]{Corollary}

\theoremstyle{definition}
\newtheorem{example}[theorem]{Example}
\newtheorem{definition}[theorem]{Definition}

\newcommand{\bc}{\mathbb{C}}

\newcommand{\bp}{\mathbb{P}}

\newcommand{\bz}{\mathbb{Z}}
\newcommand{\ca}{\mathcal{A}}
\newcommand{\cf}{\mathcal{F}}

\newcommand{\z}{\mathcal{Z}}
\newcommand{\modm}{\mathcal{M}}

\newcommand{\res}{\mathop{\mathrm{Res}}}

\begin{document}

\title{Orbifold Hurwitz numbers and Eynard--Orantin invariants}

\author{Norman Do}
\address{School of Mathematical Sciences, Monash University, Victoria 3800, Australia}
\email{\href{mailto:norm.do@monash.edu}{norm.do@monash.edu}}

\author{Oliver Leigh}
\address{Department of Mathematics and Statistics, The University of Melbourne, Victoria 3010, Australia}
\email{\href{mailto:oleigh@gmail.com}{oleigh@student.unimelb.edu.au}}

\author{Paul Norbury}
\address{Department of Mathematics and Statistics, The University of Melbourne, Victoria 3010, Australia}
\email{\href{mailto:pnorbury@ms.unimelb.edu.au}{pnorbury@ms.unimelb.edu.au}}

\thanks{This work was partially supported by the Australian Research Council grant DP1094328.}
\subjclass[2010]{32G15; 14N35; 05A15}

\begin{abstract}
We prove that a generalisation of simple Hurwitz numbers due to Johnson, Pandharipande and Tseng satisfy the topological recursion of Eynard and Orantin.  This generalises the Bouchard--Mari\~no conjecture and places Hurwitz--Hodge integrals, which arise in the Gromov--Witten theory of target curves with orbifold structure, in the context of the Eynard--Orantin topological recursion.
\end{abstract}

\maketitle

\setlength{\parskip}{0pt}
\tableofcontents
\setlength{\parskip}{6pt}

\section{Introduction} \label{introduction}

The topological recursion of Eynard and Orantin produces invariants of a Riemann surface $C$ equipped with two meromorphic functions $x: C \to \mathbb{C}$ and $y: C \to \mathbb{C}$~\cite{EOrInv}. For integers $g \geq 0$ and $n > 0$, the Eynard--Orantin invariant $\omega_{g,n}$ is a multidifferential --- in other words, a tensor product of meromorphic 1-forms --- on the Cartesian product $C^n$.  See Section~\ref{sec:EO} for a precise definition of the topological recursion and further details.

For various choices of spectral curve $(C, x, y)$, the Eynard--Orantin invariants store intersection numbers on moduli spaces of curves~\cite{EynRec,EOrInv}; Weil--Petersson volumes~\cite{EOrWei}; simple Hurwitz numbers~\cite{BEMSMat,BMaHur,EMSLap}; the enumeration of lattice points in moduli spaces of curves~\cite{NorStr}; Gromov--Witten invariants of $\bp^1$~\cite{DOSSIde, NScGro}; and conjecturally, the open and closed Gromov--Witten invariants of toric Calabi--Yau threefolds~\cite{BKMPRem,MarOpe,EOrCom}. The main result of this paper adds to this list an infinite family of examples, which generalise the relation between simple Hurwitz numbers and Eynard--Orantin invariants known as the Bouchard--Mari\~no conjecture~\cite{BMaHur}.  The methods herein generalise the proof of this result by Eynard, Mulase, and Safnuk~\cite{EMSLap}.

A great deal of attention in the literature has been paid to simple Hurwitz numbers and their relation to various moduli spaces~\cite{BMaHur, ELSV2,EMSLap, HurRie, OPaGrom}. The simple Hurwitz number $H_{g;\mu}$ is the weighted count of connected genus $g$ branched covers of $\mathbb{P}^1$ with ramification profile over $\infty$ given by the partition $\mu = (\mu_1, \mu_2, \ldots, \mu_n)$ and simple ramification elsewhere.  In this paper, we consider the generalisation $H^{[a]}_{g;\mu}$, which is the weighted count of connected genus $g$ branched covers of $\mathbb{P}^1$ with ramification profile over $\infty$ given by $\mu$, ramification profile over 0 given by a partition of the form $(a, a, \ldots, a)$, and simple ramification elsewhere.  We refer to these as {\em orbifold Hurwitz numbers} and note that we recover the simple Hurwitz numbers in the case $a = 1$. See Section~\ref{sec:hur} for a precise definition of orbifold Hurwitz numbers.

Assemble the orbifold Hurwitz numbers into the following generating function.
\begin{equation} \label{orb_gen_func}
H^{[a]}_{g,n}(x_1, \ldots, x_n) = \sum_{\mu_1, \ldots, \mu_n = 1}^\infty H^{[a]}_{g;\mu}~ \frac{|\mathrm{Aut}~\mu|}{(2g-2+n+\frac{|\mu|}{a})!} ~x_1^{\mu_1} \cdots x_n^{\mu_n}
\end{equation}
We use the notation $|\mu|$ to denote the sum $\mu_1 + \mu_2 + \cdots + \mu_n$ and $\mathrm{Aut}~\mu$ to denote the group of permutations that leave the tuple $(\mu_1, \mu_2, \ldots, \mu_n)$ invariant.

\begin{theorem} \label{th:main}
For any positive integer $a$, consider the rational spectral curve $C$ given by
\[
x(z) = z\exp(-z^a) \qquad \text{and} \qquad y(z) = z^a.
\]
The analytic expansion of the Eynard--Orantin invariant $\omega_{g,n}$ of $C$ around $x_1 = x_2 = \cdots = x_n = 0$ is given by
\begin{equation} \label{orb_gen_form}
\Omega_{g,n}^{[a]}(x_1,\ldots,x_n) = \frac{\partial}{\partial x_1} \cdots \frac{\partial}{\partial x_n} H^{[a]}_{g,n}(x_1, \ldots, x_n)~dx_1 \otimes \cdots \otimes dx_n.
\end{equation}
\end{theorem}

Here and throughout the paper, we consider $z_k$ to be the rational parameter on the $k$th copy of $C$ in the Cartesian product $C^n$ and we adopt the shorthand $x_k$ for $x(z_k)$. For notational convenience, we usually omit the $\otimes$ symbol in the tensor product of 1-forms.

Our results are motivated by and reliant upon the work of Johnson, Pandharipande and Tseng concerning Hurwitz--Hodge integrals~\cite{JPT}. They describe spaces of admissible covers and interpret them as moduli spaces of stable maps. This allows for virtual localisation to be used, which leads to the following relation between orbifold Hurwitz numbers and intersection numbers on moduli spaces of stable maps to the classifying stack ${\mathcal B} \mathbb{Z}_a$ given by a point with trivial $\mathbb{Z}_a$ action.
\[
H^{[a]}_{g,\mu} =  \frac{(2g-2+n+\frac{|\mu|}{a})!}{|\mathrm{Aut}~\mu|}\, a^{1-g+\sum \{ \mu_i / a \}} \prod_{i=1}^n \frac{\mu_i^{\lfloor \mu_i / a\rfloor}}{\lfloor \mu_i/a \rfloor!}\int_{\overline{\mathcal{M}}_{g,[-\mu]}(\mathcal{B}\mathbb{Z}_a)} \frac{\sum_{i=0}^\infty (-a)^i \lambda_i^U }{\prod_{i=1}^n (1- \mu_i \bar{\psi_i})}
\]
where $r=\lfloor r\rfloor+\{r\}$ gives the integer and fractional parts of the real number $r$.
The case $a = 1$ is the famous ELSV formula, which expresses simple Hurwitz numbers as intersection numbers on moduli spaces of curves~\cite{ELSV2}. See Section~\ref{sec:ELSV} for a precise definition of the moduli spaces and characteristic classes appearing in the formula above.

One consequence of this ELSV-type formula is that the so-called Hurwitz--Hodge integrals appearing on the right hand side can be calculated from the knowledge of the orbifold Hurwitz numbers. So our main theorem places Hurwitz--Hodge integrals, which arise in the Gromov--Witten theory of target curves with orbifold structure, in the context of Eynard--Orantin topological recursion.

Sections~\ref{sec:hur}, \ref{sec:ELSV} and \ref{sec:EO} contain preparatory material.  The heart of the proof of Theorem~\ref{th:main} is contained in Section~\ref{sec:proof}.
In Section~\ref{sec:apps}, we conclude the paper with some applications of our main theorem. In particular, the general theory of Eynard--Orantin topological recursion involves string and dilaton equations. In the context of orbifold Hurwitz numbers, these lead to relations between Hurwitz--Hodge integrals on $\overline{\mathcal M}_{g,n+1}(\mathcal{B}\mathbb{Z}_a)$ and those on $\overline{\mathcal M}_{g,n}(\mathcal{B}\mathbb{Z}_a)$.





\section{Hurwitz numbers}  \label{sec:hur}

Hurwitz numbers count branched covers $\Sigma \to \Sigma'$ of Riemann surfaces with specified branch points and ramification profiles.  Variants may require further conditions to be satisfied --- for example, that the branched covers be connected or that the preimages of a branch point be labelled. Two branched covers $\Sigma_1$ and $\Sigma_2$ are isomorphic if there exists an isomorphism of Riemann surfaces $f:\Sigma_1\to\Sigma_2$ that covers the identity on $\Sigma'$. Similarly, an automorphism of a branched cover $\Sigma \to \Sigma'$ is an automorphism of the Riemann surface $\Sigma$ that covers the identity on $\Sigma'$. For a degree $d$ branched cover, the ramification profile at a branch point is given by a partition $\mu = (\mu_1, \mu_2, \ldots, \mu_\ell)$ consisting of non-increasing positive integers that sum to $d$. Suppose that we fix $\{p_1, p_2, \ldots, p_r\} \subset \Sigma'$ together with partitions $\mu^{(1)}, \mu^{(2)}, \ldots, \mu^{(r)}$ of $d$.  The associated {\em Hurwitz number} is the weighted count of branched covers $\pi: \Sigma \to \Sigma'$ with ramification profile $\mu^{(i)}$ at $p_i$, where the weight of a branched cover $\pi$ is $\frac{1}{|\mathrm{Aut}~\pi|}$. There are two distinct flavours of Hurwitz theory corresponding to the enumerations of connected covers and disconnected covers. Although our primary goal is to understand connected Hurwitz numbers, it is often necessary to deal with disconnected Hurwitz numbers.

The Riemann existence theorem allows a branched cover of $\mathbb{P}^1$ to be described by its monodromy at the branch points. It follows that a disconnected Hurwitz number is equal to $\frac{1}{|\mu|!}$ multiplied by the number of tuples $(\sigma_1, \sigma_2, \ldots, \sigma_r)$ of permutations in the symmetric group $S_{|\mu|}$ such that
\begin{itemize}
\item $\sigma_1 \sigma_2 \cdots \sigma_r = (1)$; and
\item $\sigma_i$ has cycle type given by the partition $\mu^{(i)}$.
\end{itemize}
One obtains connected Hurwitz numbers by further requiring that the permutations $\sigma_1, \sigma_2, \ldots, \sigma_r$ generate a transitive subgroup of $S_{|\mu|}$.

\begin{definition}
For a positive integer $a$, let the {\em orbifold Hurwitz number} $H_{g;\mu}^{[a]}$ be the weighted count of connected genus $g$ branched covers of $\mathbb{P}^1$ such that
\begin{itemize}
\item the ramification profile over $\infty$ is given by the partition $\mu = (\mu_1, \mu_2, \ldots, \mu_n)$;
\item the ramification profile over 0 is given by a partition of the form $(a, a, \ldots, a)$; and
\item the only other ramification is simple and occurs at $m$ fixed points.
\end{itemize}
\end{definition}
The Hurwitz number is zero unless $a$ divides $|\mu|$ and the Riemann--Hurwitz formula implies that $m$ must be equal to $2g - 2 + n + \frac{|\mu|}{a}$. We will consistently use $m$ to denote the expression $2g - 2 + n + \frac{|\mu|}{a}$ throughout the paper. We also consider $a$ to be fixed and often drop the superscript $[a]$.

\subsection{Cut-and-join recursion}

The cut-and-join recursion provides a simple method for the calculation of Hurwitz numbers. It was originally conceived for the case of simple Hurwitz numbers~\cite{GouldenJackson} and has since been generalised in various ways~\cite{SSZDouble, Zhang-Zhou}. The basic premise is to determine the behaviour of a branched cover as two of the branch points come together. For simple ramification, this translates into understanding the behaviour of permutations multiplied by transpositions. We state the cut-and-join recursion using the following normalisation of the simple Hurwitz numbers, where $\mu = (\mu_1, \mu_2, \ldots, \mu_n)$ and $m = 2g-2+n+|\mu|$.
\[
H_g(\mu) = H_{g;\mu} \times \frac{|\mathrm{Aut}~\mu|}{m!} \qquad \qquad \qquad
\]

\begin{proposition} [Cut-and-join recursion for simple Hurwitz numbers~\cite{GJVNum}] \label{th:simrec}
The normalised simple Hurwitz numbers satisfy the following recursion, where $m=2g-2+n+|\mu|$. We use the notation $S = \{1, 2, \ldots, n\}$ and $\mu_I = (\mu_{i_1}, \mu_{i_2}, \ldots, \mu_{i_N})$ for $I = \{i_1, i_2, \ldots, i_N\}$.
\begin{align} \label{eq:simrec}
m H_g(\mu_S) &= \sum_{i < j} (\mu_i + \mu_j) H_g(\mu_{S \setminus \{i,j\}}, \mu_i+\mu_j) \\
&+ \sum_{i=1}^n \sum_{\alpha+\beta=\mu_i} \frac{\alpha \beta}{2} \Bigg[ H_{g-1}(\mu_{S \setminus \{i\}},\alpha,\beta) + \mathop{\sum_{g_1+g_2=g}}_{I \sqcup J = S \setminus \{i\}} H_{g_1}(\mu_I, \alpha) H_{g_2}(\mu_J, \beta) \Bigg] \nonumber
\end{align}
\end{proposition}



We can state the cut-and-join recursion for orbifold Hurwitz numbers using the analogous normalisation, where $\mu = (\mu_1, \mu_2, \ldots, \mu_n)$ and $m = 2g-2+n+\frac{|\mu|}{a}$.
\[
H_g^{[a]}(\mu) = H_{g;\mu}^{[a]} \times \frac{|\text{Aut } \mu|}{m!}
\]

\begin{proposition}[Cut-and-join recursion for orbifold Hurwitz numbers]  \label{th:rec}
For $\mu=(\mu_1, \mu_2, \ldots, \mu_n)$, the normalised orbifold Hurwitz numbers satisfy the cut-and-join recursion~\eqref{eq:simrec}, where $m=2g-2+n+\frac{|\mu|}{a}$.
\end{proposition}

In Appendix~\ref{sec:graphrep}, we provide a proof of this proposition via a graphical representation of branched covers. An immediate use of Proposition~\ref{th:rec} is the calculation of the generating functions appearing in Equations~\ref{orb_gen_func} and~\ref{orb_gen_form} in the base cases $(g,n) = (0,1)$ and $(g,n) = (0,2)$.

\begin{lemma} \label{th:01}
The generating function $H^{[a]}_{0,1}(x)$ satisfies the following equations.
\[
x\frac{d}{d x}H^{[a]}_{0,1} = z^a \qquad \qquad \Omega^{[a]}_{0,1} = z^{a-1}(1-az^a)~dz.
\]
\end{lemma}

\begin{proof}
In the case $(g,n) = (0,1)$, Proposition~\ref{th:rec} states that
\[
\left(\frac{\mu}{a} - 1\right) H_0(\mu) = \sum_{\alpha + \beta = \mu} \frac{\alpha \beta}{2} H_0(\alpha) H_0(\beta).
\]
This may be equivalently expressed at the level of generating functions in the following way.
\[
\frac{x}{a} \frac{dH^{[a]}_{0,1}}{dx} - H^{[a]}_{0,1} = \frac{x^2}{2} \left( \frac{dH^{[a]}_{0,1}}{dx} \right)^2 \qquad \Rightarrow \qquad \frac{z}{a(1-az^a)}\frac{dH^{[a]}_{0,1}}{dz} - H^{[a]}_{0,1} = \frac{z^2}{2(1-az^a)^2}\left(\frac{dH^{[a]}_{0,1}}{dz}\right)^2
\]
This is satisfied by $H^{[a]}_{0,1} = \frac{z^a}{a}-\frac{z^{2a}}{2}$, from which we immediately obtain the desired result.
\end{proof}


\begin{lemma}
The bidifferential $\Omega^{[a]}_{0,2}$ satisfies the following equation.
\[
\Omega^{[a]}_{0,2}=\frac{dz_1~dz_2}{(z_1-z_2)^2}-\frac{dx_1~dx_2}{(x_1-x_2)^2}.
\]
\end{lemma}

\begin{proof}
In the case $(g,n) = (0,2)$, Proposition~\ref{th:rec} states that at the level of generating functions,
\[
\left[\frac{1}{a}(1-az_1^a)x_1\frac{\partial}{\partial x_1}+\frac{1}{a}(1-az_2^a)x_2\frac{\partial}{\partial x_2}\right] H^{[a]}_{0,2}(x_1,x_2) = \frac{x_2z_1^a-x_1z_2^a}{x_1-x_2}.
\]
This is a special case of Proposition~\ref{th:PDE} below. Given that $H^{[a]}_{0,2}(x_1, x_2)$ must be symmetric in $x_1$ and $x_2$, the equation is satisfied by
\begin{equation} \label{eq:unst}
x_1\frac{\partial}{\partial x_1}H^{[a]}_{0,2}(x_1,x_2)=\frac{x_2}{x_2-x_1}-\frac{z_2}{(z_2-z_1)(1-az_1^a)}.
\end{equation}
Now apply $x_2\frac{\partial}{\partial x_2}=\frac{z_2}{1-az_2^a}\frac{\partial}{\partial z_2}$ to this equation to obtain
\[
x_1x_2\frac{\partial^2}{\partial x_1\partial x_2}H^{[a]}_{0,2}(x_1,x_2)=-\frac{x_1x_2}{(x_1-x_2)^2}+\frac{z_1x_2}{(z_1-z_2)^2(1-az_1^a)} \frac{dz_2}{dx_2} =-\frac{x_1x_2}{(x_1-x_2)^2}+\frac{x_1x_2}{(z_1-z_2)^2} \frac{dz_1}{dx_1} \frac{dz_2}{dx_2},
\]
from which we immediately obtain the desired result.
\end{proof}

The following proposition expresses the cut-and-join recursion in terms of the orbifold Hurwitz number generating functions and generalises Theorem~4.4 from~\cite{GJVNum}.  

\begin{proposition}  \label{th:PDE}
For $2g-2+n>1$, the orbifold Hurwitz number generating functions satisfy the following partial differential equation. We use the notation $S = \{1, 2, \ldots, n\}$ and $x_I = (x_{i_1}, x_{i_2}, \ldots, x_{i_N})$ for $I = \{i_1, i_2, \ldots, i_N\}$. 
\begin{align}\label{eq:genrec1}
\Bigg[ 2g-&2+n+\frac{1}{a}\sum_{i=1}^n(1-az_i^a)x_i\frac{\partial}{\partial x_i} \Bigg] H^{[a]}_{g,n}(x_{S}) \\ \nonumber
&= \sum_{i< j} \frac{1}{z_i-z_j} \left( \frac{z_j}{1-az_i^a}x_i\frac{\partial}{\partial x_i}-\frac{z_i}{1-az_j^a}x_j\frac{\partial}{\partial x_j} \right) \left[ H^{[a]}_{g,n-1}(x_{S \setminus \{j\}})+H^{[a]}_{g,n-1}(x_{S \setminus \{i\}}) \right]\\ \nonumber
&+\frac{1}{2}\sum_{i=1}^n \left[u_1u_2\frac{\partial^2}{\partial u_1\partial u_2}H^{[a]}_{g-1,n+1}(u_1, u_2, z_{S \setminus \{i\}}) \right]_{\substack{u_1=x_i \\ u_2=x_i}} \\ \nonumber
&+ \frac{1}{2} \sum_{i=1}^n \mathop{\sum^{\mathrm{stable}}_{g_1+g_2=g}}_{I\sqcup J=S \setminus \{i\}} \left[ x_i\frac{\partial}{\partial x_i} H^{[a]}_{g_1,|I|+1}(x_i,x_I) \right] \left[ x_i\frac{\partial }{\partial x_i}H^{[a]}_{g_2,|J|+1}(x_i,x_J) \right]
\end{align}
The final summation is stable in the sense that we omit all terms which involve $H^{[a]}_{0,1}$ or $H^{[a]}_{0,2}$.
\end{proposition}

\begin{proof}
Apply the operator
\[
\sum_{\mu_1, \ldots, \mu_n = 1}^\infty \big[ \cdot \big]~x_1^{\mu_1} \cdots x_n^{\mu_n}
\]
to both sides of the cut-and-join recursion to obtain the following partial differential equation satisfied by the expansion of $H^{[a]}_{g,n}$ around $x_1= \cdots = x_n = 0$.
\begin{align} \label{eq:genrec0}
\Bigg[ 2g-&2+n + \frac{1}{a}\sum_{i=1}^nx_i\frac{\partial}{\partial x_i}\Bigg] H^{[a]}_{g,n}(x_{S}) = \sum_{i< j}\frac{x_ix_j}{x_i-x_j}\left(\frac{\partial}{\partial x_i}-\frac{\partial}{\partial x_j}\right) \left[ H^{[a]}_{g,n-1}(x_{S \setminus \{j\}}) + H^{[a]}_{g,n-1}(x_{S \setminus \{i\}}) \right] \\ \nonumber
&+\sum_{i \neq j} \left[ x_i\frac{\partial}{\partial x_i} H^{[a]}_{g,n-1}(x_{S \setminus \{j\}}) \right] \left[ x_i\frac{\partial}{\partial x_i}H^{[a]}_{0,2}(x_i,x_j) \right] + \sum_{i=1}^n \left[ x_i\frac{\partial}{\partial x_i}H^{[a]}_{g,n}(x_{S}) \right] \left[ x_i\frac{\partial}{\partial x_i}H^{[a]}_{0,1}(x_i) \right] \\ \nonumber
&+ \frac{1}{2} \sum_{i=1}^n \left[ u_1u_2\frac{\partial^2}{\partial u_1\partial u_2}H^{[a]}_{g-1,n+1}(u_1,u_2,x_{S \setminus \{i\}}) \right]_{\substack{u_1=x_i \\ u_2=x_i}} \\ \nonumber
&+ \frac{1}{2}\sum_{i=1}^n\ \mathop{\sum^{\mathrm{stable}}_{g_1+g_2=g}}_{I\sqcup J=S \setminus \{i\}} \left[ x_i\frac{\partial}{\partial x_i}H^{[a]}_{g_1,|I|+1}(x_i,x_I) \right] \left[ x_i\frac{\partial}{\partial x_i}H^{[a]}_{g_2,|J|+1}(x_i,x_J) \right]
\end{align}
We have used here the fact that
\[
\frac{x_1x_2}{x_1-x_2}\left( \frac{\partial}{\partial x_1}-\frac{\partial}{\partial x_2} \right) \left[ x_1^N+x_2^N \right] = N(x_1^{N-1}x_2+x_1^{N-2}x_2^2+ \cdots +x_1x_2^{N-1}).
\]

Substitute the expression $x_i\frac{\partial}{\partial x_i}H^{[a]}_{0,2}(x_i,x_j)=\frac{x_j}{x_j-x_i}-\frac{z_j}{(z_j-z_i)(1-az_i^a)}$ from equation \eqref{eq:unst} into equation \eqref{eq:genrec0} to get cancellation of all terms involving $\frac{x_ix_j}{x_i-x_j}$. Furthermore, move the terms involving $x_i\frac{\partial}{\partial x_i}H_0(x_i) = z_i^a$ calculated in Lemma~\ref{th:01} to the left hand side in order to obtain the desired result.
\end{proof}

Many of the cancellations in the proof of Proposition~\ref{th:PDE} do not occur in the special cases $(g,n) = (0,3)$ and $(g,n) = (1,1)$. Nevertheless, equation~\eqref{eq:genrec0} is still satisfied in these special cases and simplifies to the following PDEs.
\begin{proposition}  \label{th:PDE1}
The orbifold Hurwitz generating functions $H^{[a]}_{0,3}$ and $H^{[a]}_{1,1}$ satisfy the following, where the subscripts in the first equation are to be interpreted modulo 3.
\begin{align}  \label{eq:PDE03}
\left[ 1+\frac{1}{a}\sum_{i=1}^3z_i\frac{\partial}{\partial z_i} \right] H^{[a]}_{0,3}(x_1,x_2,x_3) &= \frac{z_1z_2z_3}{(z_1-z_2)(z_2-z_3)(z_3-z_1)} \sum_{i=1}^3 \frac{z_{i-1}-z_{i+1}}{z_i(1-az_i^a)^2}  \\
\left[ 1+\frac{1}{a}z_1\frac{\partial}{\partial z_1} \right] H^{[a]}_{1,1}(x_1) &= \frac{1}{2}\frac{x_1^2}{dx_1~dx_1}\left.\left( \frac{dx_1~dx_2}{(x_1-x_2)^2}-\frac{dz_1~dz_2}{(z_1-z_2)^2} \right)\right|_{z_2=z_1}.
 \label{eq:PDE11}
\end{align}
\end{proposition}
Equations (\ref{eq:PDE03}) and  (\ref{eq:PDE11}) uniquely determine $H^{[a]}_{0,3}(x_1,x_2,x_3)$ and $H^{[a]}_{1,1}(x_1)$.  In (\ref{eq:PDE11}), the right hand side is meromorphic at $z_1=\alpha$ with no other poles, by (\ref{eq:princ02}).

\subsection{Double Hurwitz numbers}

An alternative proof of the cut-and-join equations arises by considering the action of transpositions in the symmetric group and so is more natural from the disconnected Hurwitz number viewpoint.  It assembles the Hurwitz numbers into the generating functions \eqref{eq:tau}.  We describe it here for completeness, incorporating the generalisation which satisfies the same cut-and-join equations as simple Hurwitz numbers.

Orbifold Hurwitz numbers are particular examples of double Hurwitz numbers, which count branched covers of $\mathbb{P}^1$ with specified ramification over two points and simple ramification elsewhere.  
\begin{definition}
The {\em double Hurwitz number} $H_{g;\mu, \nu}$ is the number of genus $g$ branched covers of $\mathbb{P}^1$ such that
\begin{itemize}
\item the ramification profile over $\infty$ is given by the partition $\mu = (\mu_1, \mu_2, \ldots, \mu_n)$;
\item the ramification profile over 0 is given by a partition $\nu = (\nu_1, \nu_2, \ldots, \nu_m)$; and
\item the only other ramification is simple and occurs over $m$ fixed points.
\end{itemize}
\end{definition}
The Hurwitz number must be zero unless $|\mu| = |\nu|$ and the Riemann--Hurwitz formula implies that $m = 2g - 2 + \ell(\mu) + \ell(\nu)$. 

Shadrin, Spitz and Zvonkine~\cite{SSZDouble} use the infinite wedge space formalism together with calculations involving characters of the symmetric group to prove that double Hurwitz numbers satisfy the cut-and-join equation.  The cut-and-join equation takes the form of a partial differential equation satisfied by the following generating function for double Hurwitz numbers here written in the special case of orbifold Hurwitz numbers.
\begin{align}  \label{eq:tau}
\mathbf{H}^{[a]\bullet}(s; p_1, p_2, \ldots) &= \sum_{m=0}^\infty \sum_{\mu} H_{g;\mu}^{[a]\bullet} \frac{s^m}{m!} p_{\mu_1} p_{\mu_2} \cdots p_{\mu_n} \\
\mathbf{H}^{[a]}(s; p_1, p_2, \ldots) &= \sum_{m=0}^\infty \sum_{\mu} H_{g;\mu}^{[a]} \frac{s^m}{m!} p_{\mu_1} p_{\mu_2} \cdots p_{\mu_n}\nonumber
\end{align}
Here, the inner summations are over all partitions $\mu$, including the empty partition.  The relation between disconnected and connected Hurwitz numbers can be succinctly stated using the well-known logarithm trick.
\[
\mathbf{H}^{[a]\bullet} = \exp \mathbf{H}^{[a]}.
\]
The cut-and-join recursion can be naturally expressed as the following partial differential equation for the generating function \eqref{eq:tau}. 
\begin{proposition}\cite{SSZDouble}
The generating function for connected orbifold Hurwitz numbers satisfies the following partial differential equation.
\[
\frac{\partial \mathbf{H}^{[a]}}{\partial s} = \frac{1}{2} \sum_{i,j=1}^\infty \left[ (i+j) p_i p_j \frac{\partial \mathbf{H}^{[a]}}{\partial p_{i+j}} + i j p_{i+j} \frac{\partial^2 \mathbf{H}^{[a]}}{\partial p_i \partial p_j} + i j p_{i+j} \frac{\partial \mathbf{H}^{[a]}}{\partial p_i} \cdot \frac{\partial \mathbf{H}^{[a]}}{\partial p_j} \right].
\]
\end{proposition}

\section{An ELSV-type formula}  \label{sec:ELSV}

For simple Hurwitz numbers we have the remarkable ELSV-formula due to Ekedahl, Lando, Shapiro and Vainshtein~\cite{ELSV1, ELSV2, GV_ELSV}.  

For $i = 1, \ldots, n$, define the line bundle $\mathcal{L}_i$ on $\overline{\mathcal{M}}_{g,n}$ whose fibre over $[(C,p_1,\ldots,p_n)]$ is the cotangent space at the marked point $p_i$.  Denote its first Chern class by $\psi_i = c_1(\mathcal{L}_i) \in H^{2}(\overline{\mathcal{M}}_{g,n};\mathbb{Q})$ and refer to these as {\em descendent classes}.

If $\pi$ is the universal curve over $\overline{\mathcal{M}}$ and $\mathcal{K}_\pi$ is its relative dualizing line bundle, then the {\em Hodge bundle} is $\mathbb{E}_{g,n} := \pi_{*} \mathcal{K}_\pi$. Then define $\lambda_k = c_k(\mathbb{E}_{g,n})\in H^{2k}(\overline{\mathcal{M}}_{g,n},\mathbb{Q})$, the $k$th Chern class, and $\Lambda= 1-\lambda_1+ \cdots + (-1)^g\lambda_g$. 

For $0<2g-2+n$ the linear {\em Hodge integrals} are the top intersection products of classes $\{ \lambda_k\}$ and $\{ \psi_i \}_{1\leq i\leq n} $ that are of the form
\[
 \int_{\overline{\mathcal{M}}_{g,n}} \lambda_k \psi_1^{j_1}\cdots \psi_n^{j_n}. 
\]
The ELSV formula expresses Hurwitz numbers using linear Hodge integrals,
\[
H_{g,\mu} = \frac{|\mathrm{Aut}~\mu|}{m!} \prod_{i=1}^{\ell(\mu)} \frac{\mu_i^{\mu_i}}{\mu_i!} \int_{\overline{\mathcal{M}}_{g,\ell(\mu)}} \frac{\Lambda}{\prod_{i=1}^{\ell(\mu)}(1- mu_i \psi_i)}.
\]
Note that simple Hurwitz numbers are well defined for $g=0$ and $\ell(\mu)>0$ so one defines the notation: 
\[
 \int_{\overline{\mathcal{M}}_{0,1}} \frac{\lambda_0}{1- \mu_1 \psi_1}  = \frac{1}{\mu_1^2}
\]
\[
 \int_{\overline{\mathcal{M}}_{0,2}} \frac{\lambda_0}{(1- \mu_1 \psi_1)(1- \mu_2 \psi_2)}  = \frac{1}{\mu_1+\mu_2}.
\]

In \cite{JPT} Johnson, Pandharipande and Tseng introduced a generalisation of the ELSV formula which uses generalisations of the descendent classes, Hodge classes and the Hodge integrals. 

For a positive integer $a$, we consider the cyclic group $\mathbb{Z}_a$ and the {\em moduli space of admissible covers} $\overline{\mathcal{A}}_{g,\gamma}(\mathbb{Z}_a)$, which is a compact moduli space introduced by Harris and Mumford~\cite{HMuKod}. 

Let  $\gamma= (\gamma_1,\ldots, \gamma_n)$ with each $\gamma_i \in \mathbb{Z}_a$ then an {\em admissible cover} is pair $[\pi, \tau]$ where 
\begin{itemize}
	\item $\pi: D\rightarrow (C,p_1,\ldots,p_n)$ is a degree $a$ finite map of complete curves,
	\item $\tau: \mathbb{Z}_a\times D \rightarrow D$ is a $\mathbb{Z}_a$-action,
\end{itemize} 
such that:
\begin{itemize}
	\item The curve $D$ is possibly disconnected and nodal.
	\item The class of curves $[C,p_1,\ldots,p_n] \in \overline{\mathcal{M}}_{g,n}$ is stable.
	\item The map $\pi$ takes non-singular points to non-singular points, and nodes to nodes.
	\item The pair $[\pi,\tau]$ restricts to a principal $\mathbb{Z}_a$-bundle over the punctured non-singular locus with monodromy $\gamma_i$ at $p_i$.
	\item Distinct branches of nodes in $D$ map to distinct branches of nodes in $C$, with equal ramification orders.
	\item The monodromies of the $\mathbb{Z}_a$ bundle at the two branches of a node lie in opposite conjugacy classes.
\end{itemize} 

This moduli space is isomorphic to the moduli space of stable maps $\overline{\mathcal{A}}_{g,\gamma}(\mathbb{Z}_a) \cong\overline{\mathcal{M}}_{g,\gamma}(\mathcal{B}\mathbb{Z}_a)$, where $\mathcal{B}\mathbb{Z}_a$ is the classifying stack of $\mathbb{Z}_a$ given by a point with trivial $\mathbb{Z}_a$ action \cite{JKiOrb}.  One can obtain this isomorphism by viewing an admissible cover as a principle $\mathbb{Z}_a$-bundle over the stack quotient $[D/\mathbb{Z}_a]$. This induces a stable map to the classifying stack.  Note that when $a=1$, $\mathbb{Z}_1 = \{0\}$, so $\overline{\mathcal{M}}_{g,(0,\ldots,0)}(\mathcal{B}\mathbb{Z}_1) \cong \overline{\mathcal{M}}_{g,n}$. 

Define descendent classes (often known as {\em ancestor} classes in analogous contexts) by the pullback of the forgetful map $\overline{\mathcal{M}}_{g,\gamma}(\mathcal{B}\mathbb{Z}_a)\to\overline{\mathcal{M}}_{g,n}$ 
\[
\bar{\psi}_i = \varepsilon^*(\psi_i) \in H^2(\overline{\mathcal{M}}_{g,\gamma}(\mathcal{B}\mathbb{Z}_a) ;\mathbb{Q}) 
\]

Let $U$ be the irreducible representation $U: \mathbb{Z}_a \rightarrow \mathbb{C}^{*}$ defined on a cyclic generator $g$ by  $U(g) = \mathrm{exp}\left(\frac{2\pi i }{a}\right)$. For each map $[f : [D/\mathbb{Z}_a] \rightarrow \mathcal{B}\mathbb{Z}_a]\in \overline{\mathcal{M}}_{g,\gamma}(\mathcal{B}\mathbb{Z}_a)$ the $\mathbb{Z}_a$-action on $D$ and the functoriality of the global sections functor gives that $H^0 (D,\omega_D)$ is a $\mathbb{Z}_a$-representation.  Associate the $U$-summand of this representation to $[f]\in\overline{\mathcal{M}}_{g,\gamma}(\mathcal{B}\mathbb{Z}_a)$ creating the vector bundle
\[
\mathbb{E}^U \rightarrow \overline{\mathcal{M}}_{g,\gamma}(\mathcal{B}\mathbb{Z}_a) 
\]
which is called a {\em generalised Hodge bundle}.  Define the {\em generalised Hodge classes} to be the Chern classes of this vector bundle 
\[
\lambda^U_{k} = c_i(\mathbb{E}^U) \in H^{2k}( \overline{\mathcal{M}}_{g,\gamma}(\mathcal{B}\mathbb{Z}_a);\mathbb{Q}).
\]

For $2g-2+n > 0$ the linear {\em Hodge integrals} over $\overline{\mathcal{M}}_{g,\gamma}(\mathcal{B}G)$ are the top intersection products of classes $\{ \lambda^U_k\}$ and $\{ \hat{\psi}_i \}_{1\leq i\leq n} $ that are of the form
\[
\int_{\overline{\mathcal{M}}_{g,\gamma}(\mathcal{B}G)} \lambda_i^U \bar{\psi}_1^{m_1}\cdots \bar{\psi}_n^{m_n}
\]

Johnson, Pandharipande and Tseng expressed the orbifold Hurwitz number $H^{[a]}_{g,\mu}$ in terms of these generalised Hodge integrals. The term {\em Hurwitz--Hodge integral} was introduced in \cite{BGP} and they have been extensively studied in the literature~\cite{CadCav, CavGen, CCIT, Tseng_RRoch}.

\begin{theorem} \label{th:JPT}~\cite[Theorem 1]{JPT}
For $\mu = (\mu_1, \ldots, \mu_n)$, the orbifold Hurwitz number $H^{[a]}_{g;\mu}$ satisfies
\[
H^{[a]}_{g,\mu} =  \frac{m!}{|\mathrm{Aut}~\mu|}\, a^{1-g+\sum \{ \mu_i/a \}} \prod_{i=1}^n \frac{\mu_i^{\lfloor \mu_i / a\rfloor}}{\lfloor \mu_i/a \rfloor!} \int_{\overline{\mathcal{M}}_{g,[-\mu]}(\mathcal{B}\mathbb{Z}_a)}\frac{\sum_{i=0}^\infty (-a)^i \lambda_i^U }{\prod_{i=1}^n (1- \mu_i \bar{\psi_i})}.
\]
where $[-\mu] = (-\mu_1~\mathrm{mod}~a, \ldots, -\mu_n~\mathrm{mod}~a)$. One can interpret the unstable cases by defining
\[
\int_{\overline{\mathcal{M}}_{g,[-\mu]}(\mathcal{B}\mathbb{Z}_a)}\frac{\sum_{i=0}^\infty (-a)^i \lambda_i^U }{1- \mu_1 \bar{\psi_1}} = 
	\left\{ \begin{array}{ll}
		\frac{1}{a} \cdot \frac{1}{\mu_1^2} & \mbox{if $\mu_1 \equiv 0 \pmod{a}$}, \\
		0 & \mbox{otherwise.}
	\end{array}\right.
\]
\[
\int_{\overline{\mathcal{M}}_{g,[-\mu])}(\mathcal{B}\mathbb{Z}_a)}\frac{\sum_{i=0}^\infty (-a)^i \lambda_i^U }{(1- \mu_1 \bar{\psi_1})(1- \mu_2 \bar{\psi_2})} = 
	\left\{ \begin{array}{ll}
		\frac{1}{a} \cdot \frac{1}{\mu_1+\mu_2} & \mbox{if $\mu_1+\mu_2\equiv 0 \pmod{a}$}, \\
		0 & \mbox{otherwise.}
	\end{array}\right.
\]
\end{theorem}

The proof of this theorem is by virtual localisation on the moduli space of maps $\overline{\mathcal{M}}_{g}(\mathbb{P}^1[a],\mu)$, where $\mathbb{P}^1[a]$ is the projective space $\mathbb{P}^1$ with an $a$-fold orbifold point at 0. In fact, Johnson, Pandaripande and Tseng proved a more general statement than Theorem~\ref{th:JPT}. Their formula can be used to give an expression for any double Hurwitz number by choosing $a$ sufficiently large. 

\subsection{Orbifold Hurwitz generating function}\label{sec:genfun}
For a partition $\mu$ define an analogue of Witten's notation by: 
\begin{align*}
\left< \lambda_r^U \bar{\tau}_{m_1} \cdots \bar{\tau}_{m_n} \right>_{g,n}^{(\mu)} :=\int_{\overline{\mathcal{M}}_{g,[-\mu]}(\mathcal{B}G)} \lambda_i^U \bar{\psi}_1^{m_1}\cdots \bar{\psi}_n^{m_n} 
\end{align*}

Rearranging gives,
\[
H^{[a]}_{g,\mu} =  \frac{m!}{|\mathrm{Aut}~\mu|}\sum_{j_1,\ldots,j_{\ell(\mu)} \in \mathbb{Z}_+} a^{1-g+\frac{|\mu|}{a}} \left( \prod_{i=1}^{\ell(\mu)} \frac{(\mu_i/a)^{\lfloor \mu_i / a\rfloor}}{\lfloor \mu_i/a \rfloor!} \mu_i ^{j_i} \right) \left< \Lambda^U \bar{\tau}_{j_1}\cdots \bar{\tau}_{j_l}\right>_{g, \ell(\mu)}^{(\mu)}.
\]

So the orbifold Hurwitz generating function (\ref{orb_gen_func}) becomes
\[
H^{[a]}_{g,n}(x_1,\ldots,x_n) =\sum_{j_1,\ldots,j_n \in \mathbb{Z}_+}  \sum_{\mu\in\bz_+^n}  a^{1-g+\frac{|\mu|}{a}}  \left< \Lambda^U \bar{\tau}_{j_1}\cdots \bar{\tau}_{j_n}\right>_{g, n}^{(\mu)} \prod_{i=1}^n \frac{(\mu_i/a)^{\lfloor \mu_i / a\rfloor}}{\lfloor \mu_i/a \rfloor!} \mu_i ^{j_i}  x_i^{\mu_i}.
\]

If $\mu \equiv \nu \pmod{a}$ then $ \left< \Lambda^U \bar{\tau}_{j_1}\cdots \bar{\tau}_{j_n}\right>_{g, n}^{(\mu)} =\left< \Lambda^U \bar{\tau}_{j_1}\cdots \bar{\tau}_{j_n}\right>_{g,n}^{(\nu)}$. So we can sum over the mod-classes of $\mu$: 

\[
H^{[a]}_{g,n}(x_1,\ldots,x_n) = a^{1-g} \sum_{j_1,\ldots,j_n \in \mathbb{Z}_+} \sum_{\beta \in (\mathbb{Z}a)^n}   \left< \Lambda^U \bar{\tau}_{j_1}\cdots \bar{\tau}_{j_n}\right>_{g, n}^{(\beta)}   \sum_{\mu\in\bz_+^n, [\mu]=\beta}  a^{\frac{|\mu|}{a}} \prod_{i=1}^n \frac{(\mu_i/a)^{\lfloor \mu_i / a\rfloor}}{\lfloor \mu_i/a \rfloor!} \mu_i ^{j_i}  x_i^{\mu_i}.
\]

Now, for $\mu\in\bz_+^n, [\mu]=\beta$ we set $b_i = \frac{\mu_i - \alpha_i}{a} = \lfloor \mu_i/a \rfloor$. The generating function becomes:

\begin{align}
H^{[a]}_{g,n}(x_1,\ldots,x_n) =& a^{1-g} \sum_{j_1,\ldots,j_n \in \mathbb{Z}_+} \sum_{\beta \in (\mathbb{Z}a)^n}   a^{\frac{|\beta|}{a} }  \left< \Lambda^U \bar{\tau}_{j_1}\cdots \bar{\tau}_{j_n}\right>_{g,n}^{(\beta)}   \,\,   \prod_{i=1}^n \,\sum_{b_i=0}^\infty \frac{ \,\,  \left(ab_i+\beta_i\right) ^{b_i+j_i}}{b!}x_{i}^{\left( ab_i+\beta_i \right)}  \nonumber \\
	=& a^{1-g} \sum_{j_1,\ldots,j_n \in \mathbb{Z}_+} \sum_{\beta \in (\mathbb{Z}a)^n}   a^{\frac{|\beta|}{a} }  \left< \Lambda^U \bar{\tau}_{j_1}\cdots \bar{\tau}_{j_n}\right>_{g,n}^{(\beta)}  \,\,    \prod_{i=1}^n \, f_{\beta_i ,j_i} (x_i)  \label{orb_gen_func_fs}
\end{align}
where
\[
f_{r,k} (x) : = \sum_{b=0}^\infty \frac{(ab +r)^{b+k}}{b!} x^{ab +r}.
\]
For $r=1,\ldots,a$ define
\[\xi_{k+1}^{(r)}(z)=x\frac{d}{dx}\xi_{k}^{(r)}(z),\quad \xi_{-1}^{(r)}(z)=\left\{\begin{array}{ll}z^r/r&r=1,\ldots,a-1\\z^a&r=a.\end{array}\right. \]

\begin{lemma} \label{exp_gen_lemma}
The function $z(x)=  \sum \limits_{b=0}^\infty \dfrac{(ab+1)^{b-1}}{b!} x^{ab+1}$ satifies the equation $x = z(x) \exp(-z(x)^a)$. Furthermore,  $\dfrac{z(x)^r}{r}=  \sum \limits_{b=0}^\infty \dfrac{(ab+r)^{b-1}}{b!} x^{ab+r}$ \,and \,$ f_{r,k} (x) =  \xi_{k}^{(r)}(z(x))$.
\end{lemma}
\begin{proof}
 We recall some properties of exponential generating functions. Let $f(x)$ and $g(x)$ be exponential generating functions for collections of labelled objects $F$ and $G$ respectively. Then
\begin{itemize}
	\item $f(x)g(x)$ is the exponential generating function for sequences $(A_F, A_G)$ where $A_F \in F$ and $A_G \in G$.
	\item $f(x)^k$ is the exponential generating function for sequences of $k$ objects from $F$.
	\item $\exp f(x)$ is the exponential generating function for sets of elements from $F$ of all cardinalities.
\end{itemize}
We now use these properties to prove the result.
\begin{itemize}
	\item $z(x)$ is the exponential generating function for cactus-node trees of type $(1,a,\ldots,a)$ (see Appendix \ref{Comb_appendix}, Definition~\ref{cactus-node_def} and Proposition~\ref{cactus-node_enum}). Removing the node of type 1 we obtain a pair $(x,C)$ where $x$ is a point representing the node and $C$ is a collection of rooted-cactus-node trees of type $(a,\ldots,a)$ .
\begin{center}
\includegraphics[scale=0.6]{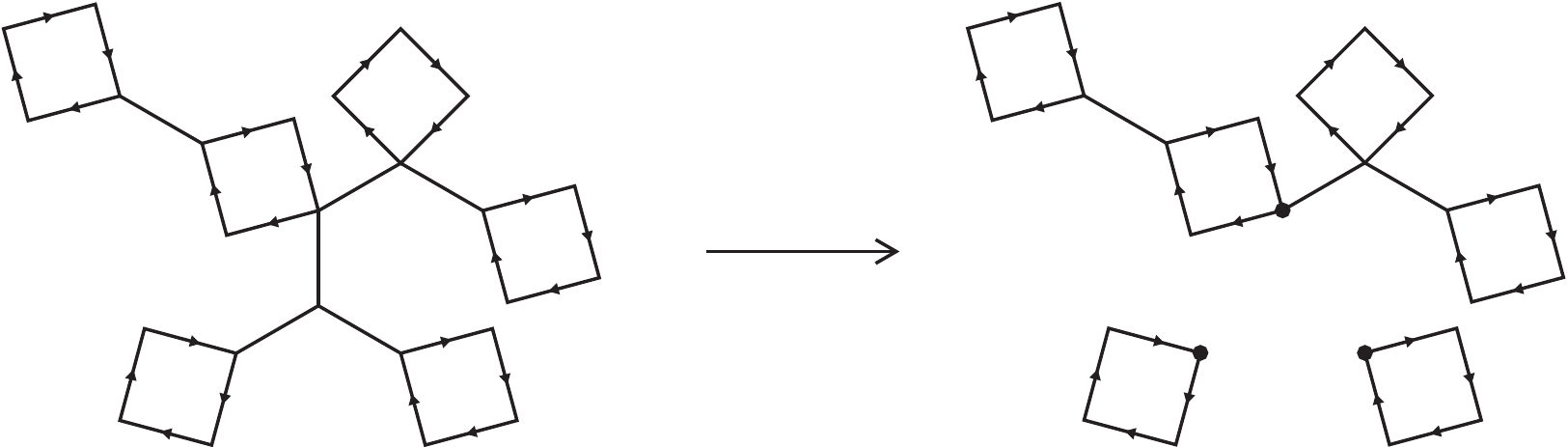}
\end{center}
A rooted-cactus-node trees of type $(a,\ldots,a)$ is a seqence $(T_1,\ldots,T_a)$ of cactus-node trees of type $(1,a,\ldots,a)$ . 
\begin{center}
\includegraphics[scale=0.6]{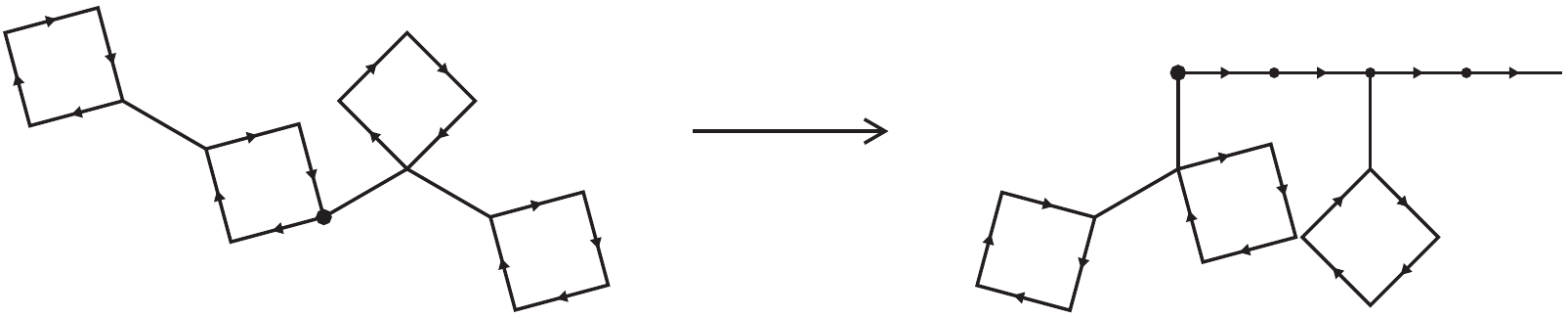}
\end{center}
Hence $z(x)$ satisfies $z(x) = x \exp(z(x)^a).$
	\item  $r \sum \limits_{b=0}^\infty \dfrac{(ab+r)^{b-1}}{b!} x^{ab+r}$ is the exponential generating function for cactus-node trees of type $(r,a,\ldots,a)$ with one of the points in the r-node marked. A cactus-node trees of this type is a sequence $(T_1,\ldots,T_r)$ of cactus-node trees of type $(1,a,\ldots,a)$.
\begin{center}
\includegraphics[scale=0.6]{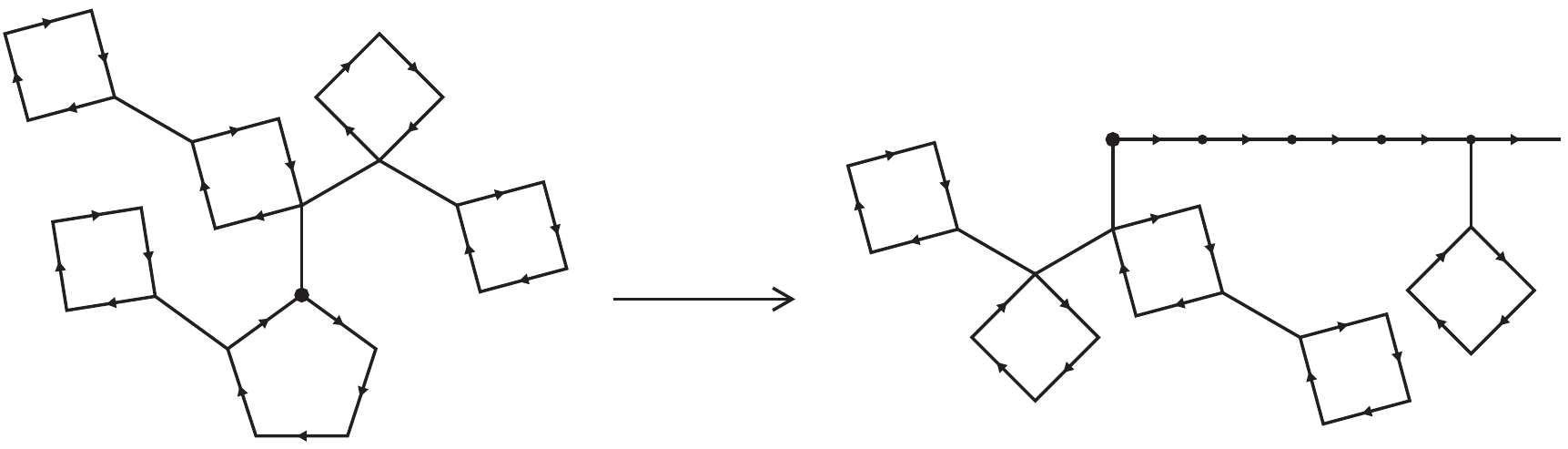}
\end{center} 
Hence,  $z(x)^r=  r \sum \limits_{b=0}^\infty \dfrac{(ab+r)^{b-1}}{b!} x^{ab+r}$.
	\item Finally, we have
\begin{align*}
f_{r,k} (x) &= \sum_{b=0}^\infty \frac{(ab +r)^{b+k}}{b!} x^{ab +r} = \left( x \frac{d}{dx} \right)^{k+1} \sum_{b=0}^\infty \frac{(ab +r)^{b-1}}{b!} x^{ab +r} \\
&= \left( x \frac{d}{dx} \right)^{k+1} \frac{z^r}{r} = \xi_{k}^{(r)}(z(x))  \qedhere
\end{align*}
\end{itemize}
\end{proof}

Define
\begin{equation}  \label{eq:Fagn}
 F^{[a]}_{g,n}(z_1,\ldots,z_n):=H^{[a]}_{g,n}(x_1(z_1),\ldots,x_n(z_n)).
\end{equation}
Applying lemma \ref{exp_gen_lemma} to the orbifold generating function (\ref{orb_gen_func_fs}) we have,
\begin{align*}
H^{[a]}_{g,n}(x_1,\ldots,x_n) =& a^{1-g} \sum_{j_1,\ldots,j_n \in \mathbb{Z}_+} \sum_{\beta \in (\mathbb{Z}a)^n}   a^{\frac{|\beta|}{a} }  \left< \Lambda^U \bar{\tau}_{j_1}\cdots \bar{\tau}_{j_n}\right>_{g,n}^{(\beta)}  \,\,    \prod_{i=1}^n \, \xi^{(\beta_i)}_{j_i} (z_i(x_i)) 
\end{align*}
which gives
\begin{align}
F^{[a]}_{g,n}(z_1,\ldots,z_n) =& a^{1-g} \sum_{j_1,\ldots,j_n \in \mathbb{Z}_+} \sum_{\beta \in (\mathbb{Z}a)^n}   a^{\frac{|\beta|}{a} }  \left< \Lambda^U \bar{\tau}_{j_1}\cdots \bar{\tau}_{j_n}\right>_{g,n}^{(\beta)}  \,\,    \prod_{i=1}^n \, \xi^{(\beta_i)}_{j_i} (z_i) \label{orb_gen_func_xis}
\end{align}
A key consequence is that the generating functions $H^{[a]}_{g,n}(x_1,\ldots,x_n)$ are {\em rational} in $(z_1,\ldots,z_n)$, or more accurately are local expansions around $x_i=0$ of rational functions.  

The function $x = z \exp(-z^a)$ defines local involutions $z\mapsto\sigma_{\alpha}(z)$ near each root $\alpha$ of $dx(\alpha)=0$.  Via a local coordinate $s_{\alpha}$ such that $x=s_{\alpha}^2+x(\alpha)$,   the involution is given by $\sigma_{\alpha}(s_{\alpha})=-s_{\alpha}$.  It gives rise to the following vector space.
\begin{definition} Define the vector space $\ca_x$ to consist of rational functions $p$ satisfying:
\begin{itemize}
\item $p$ has poles only at $\{\alpha: dx(\alpha)=0\}$;
\item $p(z)+p(\sigma_{\alpha}(z))$ is analytic at $z=\alpha$.
\end{itemize}
\end{definition}

\begin{lemma}  \label{th:xinax}
For $k\geq 0$ and $r=1,\ldots,a$, $\xi_{k}^{(r)}(z)\in\ca_x$ and form a basis.
\end{lemma}
\begin{proof}
The proof uses the simple fact that $x\frac{d}{dx}$ preserves $\ca_x$, i.e. 
\begin{equation}   \label{eq:pres}
x\frac{d}{dx}\ca_x\subset\ca_x.
\end{equation}
This can be seen as follows.  Clearly $x\frac{d}{dx}=\frac{z}{1-az^a}\frac{d}{dz}$ introduces no new poles outside $\{\alpha: dx(\alpha)=0\}$.  So we study its behaviour locally around a single pole $\alpha$.  The principal part of any function in $\ca_x$ is an odd polynomial in $s_{\alpha}^{-1}$ for $s_{\alpha}$ the local coordinate defined above (since $z\mapsto\sigma_{\alpha}(z)$ corresponds to $s_{\alpha}\mapsto -s_{\alpha}$), and
\[ x\frac{d}{dx}=\frac{s_{\alpha}^2+\alpha}{2s_{\alpha}}\frac{d}{ds_{\alpha}}\]
maps odd polynomials in $s_{\alpha}^{-1}$ to odd polynomials in $s_{\alpha}^{-1}$ since it preserves the parity of the power of any monomial in $s_{\alpha}$.  Furthermore, $x\frac{d}{dx}\bc[[s_{\alpha}]]\subset\bc s_{\alpha}^{-1}\oplus\bc[[s_{\alpha}]]\subset\ca_x$.  Hence \eqref{eq:pres} is proven.

Now $\xi_{-1}^{(r)}(z)=z^r/r\in\ca_x$ (or $z^a$ for $r=a$) since it is analytic at $z=\alpha$.  Since $\xi_{k}^{(r)}(z)=x\frac{d}{dx}\xi_{k-1}^{(r)}(z)$ and $x\frac{d}{dx}$ preserves $\ca_x$ ,by induction $\xi_{k}^{(r)}(z)\in\ca_x$ for all $r$ and $k$.

A simple dimension argument proves that the $\xi_{k}^{(r)}(z)$ form a basis.
\end{proof}
{\em Remark.}  If $f(z)$ is analytic at $z=\alpha$ and satisfies $f(z)=f(\sigma_{\alpha}(z))$ then $x\frac{d}{dx}f(z)$ is analytic at $z=\alpha$.  In the terminology of the proof of Lemma~\ref{th:xinax}, the local expansion of $f(z)$ lies in $\bc[[s_{\alpha}^2]]$ and $x\frac{d}{dx}\bc[[s_{\alpha}^2]]\subset \bc[[s_{\alpha}^2]]$.

Since each $\xi_k^{(r)}(z)\in\ca_x$ we have proven:

\begin{corollary}  \label{th:fagnsym}
For $2g-2+n>0$ and all $a$, $F^{[a]}_{g,n}(z_1,\ldots,z_n)\in\ca_{x_i}$ for each $i=1,\ldots,n$.
\end{corollary}

\section{Eynard--Orantin invariants}  \label{sec:EO}

Consider a triple $(C,x,y)$ consisting of a genus 0 Riemann surface $C$ and meromorphic functions  $x,y:C\to\bc$ with the property that the zeros of $dx$ are simple and disjoint from the zeros of $dy$.   For every $(g,n)\in\bz^2$ with $g\geq 0$ and $n>0$ the Eynard--Orantin invariant of   $(C,x,y)$ is  a multidifferential $\omega^g_n(p_1,\ldots,p_n)$, i.e. a tensor product of meromorphic 1-forms on the product $C^n$, where $p_i\in C$.  (More generally, if $C$ has positive genus it should come with  a {\em Torelli marking} which is a choice of symplectic basis $\{a_i, b_i\}_{i=1,\ldots,g}$ of the first homology group $H_1(\bar{C})$ of the compact closure $\bar{C}$ of $C$.  In particular, a genus 0 surface $C$ requires no Torelli marking.)  When $2g-2+n>0$, $\omega^g_n(p_1,\ldots,p_n)$ is defined recursively in terms of local  information around the poles of $\omega^{g'}_{n'}(p_1,\ldots,p_n)$ for $2g'+2-n'<2g-2+n$.  Equivalently, the $\omega^{g'}_{n'}(p_1,\ldots,p_n)$ are used as kernels on the Riemann surface.   

Since each zero $\alpha$ of $dx$ is simple, for any point $p\in C$ close to $\alpha$ there is a unique point $\hat{p}\neq p$ close to $\alpha$ such that $x(\hat{p})=x(p)$.  The recursive definition of $\omega^g_n(p_1,\ldots,p_n)$ uses only local information around zeros of $dx$ and makes use of the well-defined map $p\mapsto\hat{p}$ there. The invariants are defined as follows.  Given a rational coordinate $z$ on $C$
\begin{align*}
\omega^0_1&=-\frac{y(z)dx(z)}{x(z)}\nonumber\\
\omega^0_2&=\frac{dz_1 \otimes dz_2}{(z_1-z_2)^2}
\end{align*}
For $2g-2+n>0$,
\begin{equation}  \label{eq:EOrec}
\omega^g_{n}(z_1,z_{S'})=\sum_{\alpha}\res_{z=\alpha}K(z_1,z) \biggr[\omega^{g-1}_{n+1}(z,\hat{z},z_{S'})+ \mathop{\sum_{g_1+g_2=g}}_{I\sqcup J=S}
\omega^{g_1}_{|I|+1}(z,z_I)\omega^{g_2}_{|J|+1}(\hat{z},z_J)\biggr]
\end{equation}
where the sum is over the zeros $\alpha$ of $dx$, $S'=\{2,\ldots,n\}$, $(g_1,|I|)\neq (0,0)\neq(g_2,|J|)$ and 
\[ K(z_1,z)=\frac{-\int^z_{\hat{z}}\omega_2^0(z_1,z')x(z)}{2(y(z)-y(\hat{z}))dx(z)}=\frac{x(z)}{2(y(\hat{z})-y(z))x'(z)}\left(  \frac{1}{z-z_1}- \frac{1}{\hat{z}-z_1}\right)\frac{dz_1}{dz}\] 
is well-defined in the vicinity of each zero of $dx$.   Note that the quotient of a differential by the differential $dx(z)$ is a meromorphic function.  The recursion \eqref{eq:EOrec} depends only on the meromorphic differential $ydx/x$ and the map $p\mapsto\hat{p}$ around zeros of $dx$.  For $2g-2+n>0$, each $\omega^g_n$ is a symmetric multidifferential with poles only at the zeros of $dx$, of order $6g-4+2n$, and zero residues.

Define $\Phi(z)$ by $d\Phi(z)=y(z)dx(z)/x(z)$.  For $2g-2+n>0$, the invariants satisfy the dilaton equation~\cite{EOrInv}:
\[
\sum_{\alpha}\res_{z=\alpha}\Phi(z)\omega^g_{n+1}(z,z_1,...,z_n)=(2-2g-n)\omega^g_n(z_1,...,z_n)
\] 
where the sum is over the zeros $\alpha$ of $dx$, $\Phi(z)=\int^z ydx(z')$ is an arbitrary antiderivative and $z_S=(z_1,\dots,z_n)$.  This enables the definition of the so-called {\em symplectic invariants}
\[ F_g=\sum_{\alpha}\res_{z=\alpha}\Phi(z)\omega^g_{1}(z).\]

{\em Remark.}  There are variations on the definition of the Eynard--Orantin invariants determined by how $x$ and $y$ appear in the kernel $K$.  Here we have used $dx/x$ and $y$ to define the kernel $K$ but any of the four combinations of $dx$ or $dx/x$ and $dy$  or $dy/y$ can be used.  All are equivalent via changes of coordinates $u=\log{x}$ and $v=\log{y}$, but in order to have $x$ appear algebraically in generating functions, the choice here suits best.

\subsection{Principal parts}

We will see below that the Eynard--Orantin recursion \eqref{eq:EOrec} which is given as a sum
\[ \omega^g_{n}(z_1,z_2,...,z_n)=\sum_{\alpha}\res_{z=\alpha}K(z_1,z)\cf(z,z_2,...,z_n)\]
over $\{\alpha:dx(\alpha)=0\}$ expresses $\omega^g_n(z_1,...,z_n)$ as the sum of its principal parts in $z_1$ at its poles $z_1=\alpha$.  This is an important feature so we explain it below after first recalling the definition and properties of principal parts.

Given a local parameter $z$ of a curve $C$, the {\em principal part} at a point $\alpha\in C$ of a function or differential $h(z)$ analytic in $U \setminus \{\alpha\}$ for some neighbourhood $U$ of $\alpha$ is 
\begin{equation}   \label{eq:PP} 
[h(z)]_{\alpha}:=\res_{w=\alpha}\frac{h(w)dw}{z-w}.
\end{equation}
(Strictly we might write $z=z(p)$ and $w=w(q)$ for points $p$ and $q$ on $C$ but we abuse terminology and identify $U$ with $z(U)$.)  It satisfies the properties:
\begin{enumerate}[(i)]
\item $[h(z)]_{\alpha}$ is analytic on $U \setminus \{\alpha\}$;

\item $h(z)-[h(z)]_{\alpha}$ is analytic on $U$.
\end{enumerate}

Thus $[h(z)]_{\alpha}$ is given by the negative part of the Laurent series of $h(z)$ at $\alpha$.

To see (i), given $z\in U \setminus \{\alpha\}$, choose a contour $\gamma_1$ around $\alpha$ not containing $z$ to calculate the residue, as in Figure~\ref{fig:contour}. 
\begin{center} 
\includegraphics[scale=0.2]{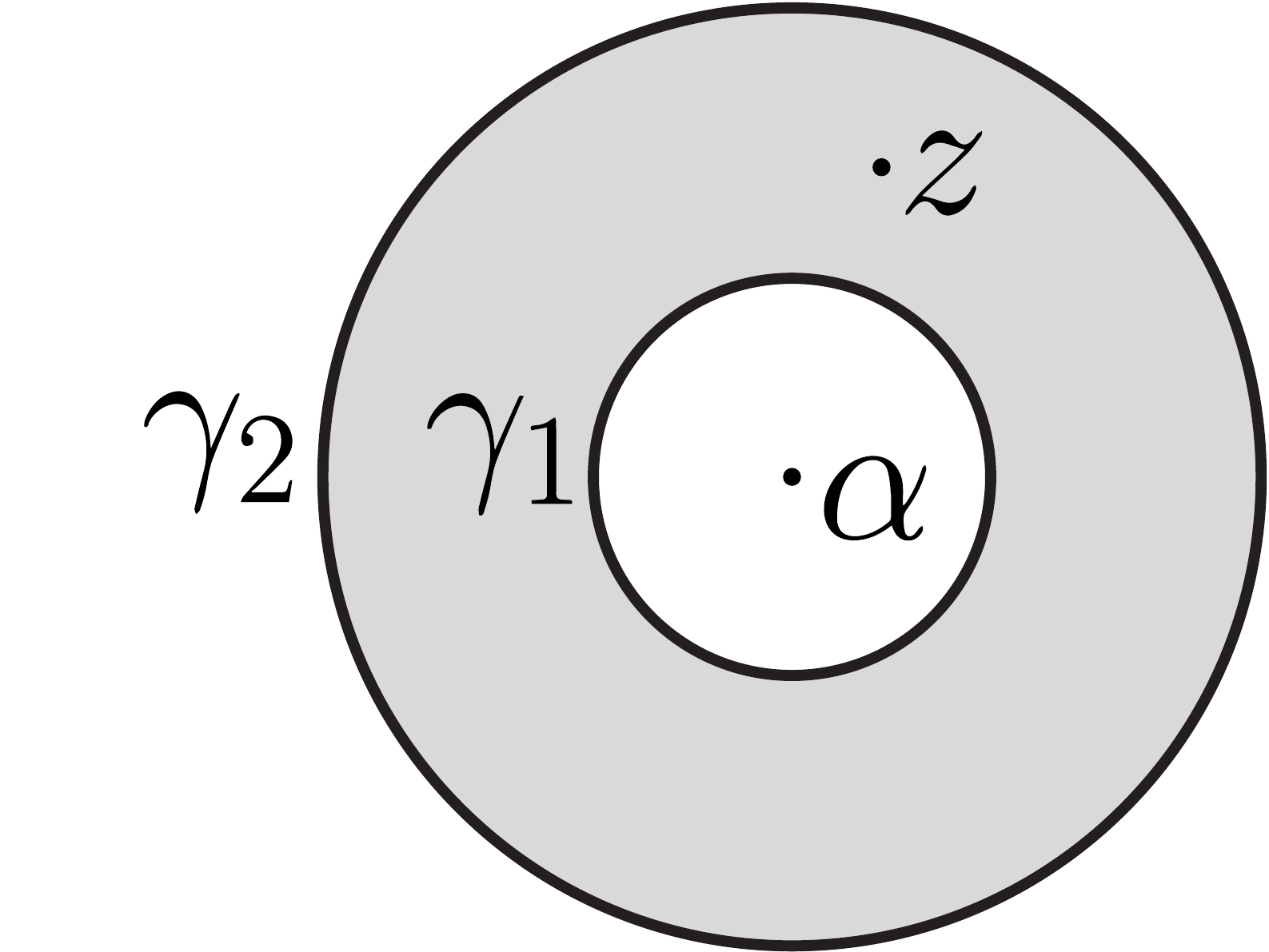}\hspace{1.1cm}\mbox{}\\
Figure \refstepcounter{figure}\arabic{figure}. \label{fig:contour}
\end{center}
Now simply differentiate under the integral sign to prove analyticity at $z$.  To see (ii), Figure~\ref{fig:contour} gives
\[ h(z)=-\res_{w=z}\frac{h(w)dw}{z-w}=\frac{1}{2\pi i}\int_{\gamma_1-\gamma_2}\frac{h(w)dw}{z-w}=[h(z)]_{\alpha}-\frac{1}{2\pi i}\int_{\gamma_2}\frac{h(w)dw}{z-w}\]
and the integral around $\gamma_2$ is analytic in $z\in U$ again since we can differentiate under the integral sign.

Suppose $C$ is a rational curve and $z$ is a rational parameter as will be true in our applications.  If $h(z)$ has a pole at $\alpha$ then $[h(z)]_{\alpha}$ is a polynomial in $1/(z-\alpha)$ (or $z$ when $\alpha=\infty$.)   Up to a constant, any rational function is the sum of its principal parts commonly known as its partial fraction decomposition.  Any rational differential is equal to the sum of its principal parts (with no constant ambiguity.)   

The principal part of a function of several variables or multidifferential $h(z_1,\ldots,z_n)$ at the point $z_1=\alpha$ is defined via \eqref{eq:PP} as if the $z_j$, $j>1$ are constants.   It is denoted $[h(z_1,\ldots,z_n)]_{z_1=\alpha}$ or $[h(z_1,\ldots,z_n)]_{\alpha}$ when $z_1$ is understood.  

For $h(z)$ analytic in $U \setminus \{b\}$
\[ \left[\frac{h(z_1)}{z_1-z_2}\right]_{z_1=\alpha}=\res_{w=\alpha}\frac{h(w)dw}{(w-z_2)(z_1-w)}\]
and we choose the contour containing $z_1$ and not $z_2$ as in Figure~\ref{fig:contour2},
\begin{center}  
\includegraphics[scale=0.2]{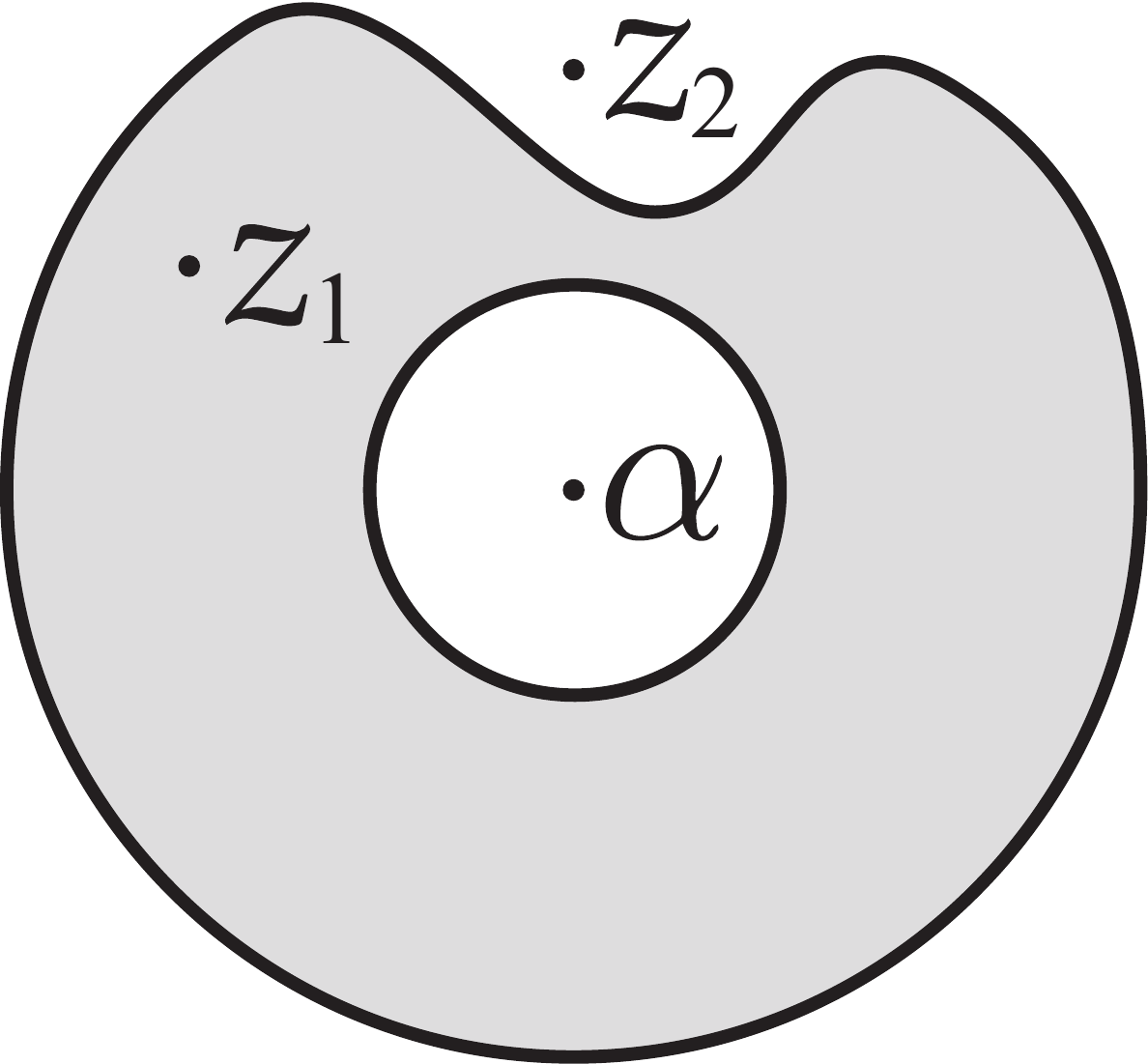}\\
Figure \refstepcounter{figure}\arabic{figure}. \label{fig:contour2}
\end{center}
to obtain the properties:
\begin{enumerate}[(i)]
\item $[h(z_1)/(z_1-z_2)]_{\alpha}$ is analytic in $z_1$ and $z_2$ on $U \setminus \{\alpha\}$;

\item $h(z_1)/(z_1-z_2)-[h(z_1)/(z_1-z_2)]_{\alpha}$ is analytic in $z_1$ on $U$.
\end{enumerate}

Note that in (ii) we allow $z_1$ to take the value $z_1=\alpha$ whereas we not allow $z_2$ to take the value $z_2=\alpha$.  Note also that:
\begin{equation}  \label{eq:princ}
\left[\frac{h(z_1)}{z_1-z_2}\right]_{z_1=\alpha}=\left[\frac{h(z_1)-h(z_2)}{z_1-z_2}\right]_{z_1=\alpha}=\left[\frac{h(z_2)}{z_2-z_1}\right]_{z_2=\alpha}.
\end{equation}
In particular, if $h(z)$ is analytic at $\alpha$ then $\displaystyle\left[\frac{h(z_1)}{z_1-z_2}\right]_{\alpha}=0$.

The principal part of a function with respect to more than one variable depends on the order and so is slightly subtle.   For example,
\[ \Bigg[\left[\frac{1}{z_1z_2}\right]_{z_1=0}\Bigg]_{z_2=0}=\Bigg[\frac{1}{z_1z_2}\Bigg]_{z_2=0}=\frac{1}{z_1z_2}\]
so we see that principal parts at 0 with respect to $z_1$ and $z_2$ commute on $1/z_1z_2$ and more generally for any product $h_1(z_1)h_2(z_2)$, whereas
\[\Bigg[\left[\frac{1}{z_1(z_1-z_2)}\right]_{z_1=0}\Bigg]_{z_2=0}=\Bigg[-\frac{1}{z_1z_2}\Bigg]_{z_2=0}=-\frac{1}{z_1z_2},\quad\quad
\Bigg[\left[\frac{1}{z_1(z_1-z_2)}\right]_{z_2=0}\Bigg]_{z_1=0}=\Big[0\Big]_{z_1=0}=0\]
so they do not commute on $1/z_1(z_1-z_2)$.   In this paper, whenever we have a function of several variables we only take the principal part with respect to one of the variables so the subtlety described here never arises.

\subsection{Principal parts of Eynard--Orantin invariants}
An important property of the stable Eynard--Orantin invariants is that they are meromorphic multidifferentials with poles at the zeros of $dx$.  In particular, on a rational curve they are rational multidifferentials and hence equal to the sum of their principal parts.  

Each summand at a zero $\alpha$ of $dx$ in the RHS of the defining recursion \eqref{eq:EOrec} for $\omega^g_{n}(z_1,\ldots,z_n)$ has dependence on $z_1$ only occurring as $1/(z-z_1)$ and $1/(\hat{z}-z_1)$.  Express \eqref{eq:EOrec} as $\omega^g_n=\sum_{\alpha} I_{\alpha}$ then:
\begin{align*}  
I_{\alpha}&=\res_{z=\alpha}\left(  \frac{dz_1}{z-z_1}- \frac{dz_1}{\hat{z}-z_1}\right)\frac{x(z)}{2(y(\hat{z})-y(z))dx(z)} \biggr[\omega^{g-1}_{n+1}(z,\hat{z},z_{S'})+ \mathop{\sum_{g_1+g_2=g}}_{I\sqcup J=S'} \omega^{g_1}_{|I|+1}(z,z_I)\omega^{g_2}_{|J|+1}(\hat{z},z_J)\biggr]
\end{align*}

One can differentiate $I_{\alpha}$ with respect to $z_1$ under the integral sign showing that it is analytic everywhere except possibly at $z_1=\alpha$.  Note that $I_{\alpha}$ is analytic at $z_1=\infty$ (assuming $\alpha\neq\infty$) since 
\[ K\sim \frac{(\hat{z}-z)x(z)}{2(y(\hat{z})-y(z))x'(z)}\frac{1}{dz}\frac{dz_1}{z_1^2}\]
so for some $C$ constant in $z_1$, $\displaystyle I_{\alpha}\sim \frac{Cdz_1}{z_1^2}$ which is analytic at $z_1=\infty$.  Thus $I_{\alpha}$ is rational and equal to its principal part at $z_1=\alpha$.  But then $I_{\alpha}$ is the principal part of $\omega^g_{n}(z_1,\ldots,z_n)$ at $z_1=\alpha$.

Furthermore, we can calculate $I_{\alpha}$ since
\begin{align*}  
I_{\alpha}&=\res_{z=\alpha}\left(  \frac{\eta(z)}{z-z_1}- \frac{\eta(\hat{z})}{\hat{z}-z_1}\right)dz_1\\
&=2\res_{z=\alpha} \frac{\eta(z)}{z-z_1}dz_1\\
&=-2[\eta(z_1)]_{\alpha}\\
&=-\bigg[\frac{x(z_1)}{(y(\hat{z}_1)-y(z_1))dx(z_1)} \Big(\omega^{g-1}_{n+1}(z_1,\hat{z}_1,z_{S'}) + \mathop{\sum_{g_1+g_2=g}}_{I\sqcup J=S'} \omega^{g_1}_{|I|+1}(z_1,z_I)\omega^{g_2}_{|J|+1}(\hat{z}_1,z_J)\Big)\bigg]_{\alpha}
\end{align*}
where $\eta(z)$ is a differential form that satisfies $\eta(z)=-\eta(\hat{z})$ (since $x(z)=x(\hat{z})$ and $\omega^{g-1}_{n+1}$ is symmetric in its arguments.)  Notice that $I_{\alpha}(z_1)=-I_{\alpha}(\hat{z}_1)$.

In summary, we have proven:
\begin{proposition}  \label{th:EOprinc}
The recursion (\ref{eq:EOrec}) expresses any Eynard-Orantin invariant as the sum of its principal parts:
\begin{equation}  \label{eq:EOprinc}
\omega^g_n(z_1,z_{S'})=-\sum_{\alpha} \bigg[\frac{x(z_1)}{(y(\hat{z}_1)-y(z_1))dx(z_1)} \Big(\omega^{g-1}_{n+1}(z_1,\hat{z}_1,z_{S'}) + \mathop{\sum_{g_1+g_2=g}}_{I\sqcup J=S'} \omega^{g_1}_{|I|+1}(z_1,z_I)\omega^{g_2}_{|J|+1}(\hat{z}_1,z_J)\Big)\bigg]_{\alpha}
\end{equation}
where the sum is over the zeros $\alpha$ of $dx$, $S'=\{2,\ldots,n\}$ and $(g_1,|I|)\neq (0,0)\neq(g_2,|J|)$.
\end{proposition}

The principal parts of the stable $\omega^g_n$ in the right hand side of (\ref{eq:EOprinc}) are straight-forward because their pole structure is rather simple.  The terms involving $\omega^0_2$ are less straight-forward so later we will require the following:
\begin{lemma}  \label{th:02inv}
\[\left[\left(\frac{dz_1}{z_1-z_2}+\frac{d\hat{z}_1}{\hat{z}_1-z_2}\right)\frac{x_1}{dx_1}\right]_{\alpha}=\left[\frac{x_1}{x_1-x_2}\right]_{\alpha}\]
\end{lemma}
\begin{proof}  
Note that we express the terms on the left hand side as quotients of differentials instead of functions for later ease.

Locally 
\begin{equation}  \label{eq:locfac}
x(z_1)-x(z_2)=(z_1-z_2)(\hat{z}_1-z_2)h(z_1,z_2)
\end{equation}
where $h(z_1,z_2)$ is analytic and non-zero at $z_1=\alpha$ and $h(z_1,z_2)=h(\hat{z}_1,z_2)$.  By the remark after Lemma~\ref{th:xinax},  since $\log h(z_1,z_2)$ is analytic at $z_1=\alpha$ and invariant under $z_1\mapsto\hat{z}_1$, then $x_1\frac{d}{dx_1}\log h(z_1,z_2)$ is analytic at $z_1=\alpha$ (where as usual $x_1=x(z_1)$.)  Hence the right hand side of
\begin{equation}  \label{eq:02anal} 
x_1\frac{d}{dx_1}\log h(z_1,z_2)=\frac{x_1}{x_1-x_2}-\left(\frac{dz_1}{z_1-z_2}+\frac{d\hat{z}_1}{\hat{z}_1-z_2}\right)\frac{x_1}{dx_1}
\end{equation}
is analytic at $z_1=\alpha$ and the lemma follows.
\end{proof}
Take the exterior of \eqref{eq:02anal} with respect to the second variable to get:
\begin{equation}  \label{eq:princ02}
\left[\left(\frac{dz_1dz_2}{(z_1-z_2)^2}+\frac{d\hat{z}_1dz_2}{(\hat{z}_1-z_2)^2}\right)\frac{x_1}{dx_1}\right]_{\alpha}=\left[\frac{x_1dx_2}{(x_1-x_2)^2}\right]_{\alpha}.
\end{equation}
and hence 
\[ \frac{dx_1dx_2}{(x_1-x_2)^2}-\frac{dz_1dz_2}{(z_1-z_2)^2}-\frac{d\hat{z}_1dz_2}{(\hat{z}_1-z_2)^2}\]
is analytic at $z_1=\alpha$ and vanishes there for all $z_2$ (in particular, for $z_2=\alpha$.)\\
\\
Notice that $x$ appears in the definition only via $dx/x$ which is rational for the spectral curve of interest in this paper
\[ x=z \exp(-z^a),\quad y=z^a.\]
Hence the kernel $K$ is also rational:
\[ K(z_1,z)=\frac{z}{2(\hat{z}^a-z^a)(1-az^a)}\left(  \frac{1}{z-z_1}- \frac{1}{\hat{z}-z_1}\right)\frac{dz_1}{dz}.\]

By the construction of Eynard--Orantin invariants as a sum of their principal parts one easily sees that $\omega^g_n$ has invariance properties under the local involutions $\sigma_{\alpha}(z)$ defined near each zero $\alpha$ of $dx$.  In the language of Section~\ref{sec:genfun} for $x=z \exp(-z^a)$, we see that
for $2g-2+n>0$, $\omega^g_n(z_1,\ldots,z_n)\in\ca_{x_i}$ for each $i=1,\ldots,n$.

{\em Remark.}  In~\cite{EynInt,EynInv} Eynard chooses a basis $d\xi_{\alpha,m}$ of $\ca_x$, linearly related to the basis $d\xi_{k}^{(r)}(z)$ in Section~\ref{sec:genfun}, and identifies the coefficients in terms of intersection numbers over a moduli space of $a$-coloured Riemann surfaces, $\overline{\modm}^a_{g,n}$
\[\omega^g_n(z_1,\ldots,z_n)=\sum_{i_1,\ldots,i_n}\sum_{d_1,\ldots,d_n}A^{(g)}_n(i_1,d_1;\ldots;i_n,d_n)\prod d\xi_{\alpha_{i_k},d_k}(z_k)\]
where $\alpha_1,\ldots,\alpha_a$ are the zeros of $dx$.  
Essentially Eynard showed that the Eynard--Orantin invariants give cohomological field theories.  This was made more precise in~\cite{DOSSIde}.
Eynard described his result as a generalised ELSV formula.  It is intriguing that the ELSV-type formula in this paper transforms linearly to Eynard's formula and hence we see a relationship between intersection numbers over $\overline{\mathcal{M}}_{g,\gamma}(\mathcal{B}\mathbb{Z}_a)$ and  intersection numbers over  $\overline{\modm}^a_{g,n}$.

\section{Proof of Main Theorem} \label{sec:proof}

In this section we prove that the Eynard--Orantin invariants of the spectral curve
\[
x=z \exp(-z^a),\quad y=z^a
\]
coincide with (the total derivatives of) the orbifold Hurwitz number generating functions.  The strategy of proof is quite natural.  Since the Eynard--Orantin recursion expresses the invariants as a sum over the principal parts in the first variable then we will analyse the principal parts of the partial differential equation \eqref{eq:genrec1}. Furthermore, the principal parts of the Eynard--Orantin invariants are antisymmetric with respect to the local involutions at each zero of $dx$, so we take the anti-symmetric part of the principal part of \eqref{eq:genrec1}.  (In fact we take the symmetric part of the principal part of \eqref{eq:genrec1} since it contains an extra anti-symmetric factor.)  This is exactly the strategy of proof used in~\cite{EMSLap}.

\begin{proof}[Proof of Theorem~\ref{th:main}]
Recall from Section~\ref{sec:ELSV} that
\[
F^{[a]}_{g,n}(z_1,\ldots,z_n):=H^{[a]}_{g,n}(x(z_1),\ldots,x(z_n))
\]
is a rational function of the $z_i$.  Equivalently, $H^{[a]}_{g,n}(x_1,\ldots,x_n)$ gives a local expansion of the rational function $F^{[a]}_{g,n}(z_1,\ldots,z_n)$ in the local coordinate $x(z_i)$ around $z_i=0$.  

Furthermore, $x = z \exp(-z^a)$ defines local involutions $\sigma_{\alpha}(z)$ around each zero $\alpha$ of $dx$ and for $2g-2+n>0$, and Corollary~\ref{th:fagnsym} gives that $F^{[a]}_{g,n}(z_1,\ldots,z_n)$ satisfies:
\begin{itemize}
\item $F^{[a]}_{g,n}(z_1,\ldots,z_n)$ has poles only at $\{z_i=\alpha: dx(\alpha)=0\}$;
\item $F^{[a]}_{g,n}(z_1,\ldots,z_i,\ldots,z_n)$+$F^{[a]}_{g,n}(z_1,\ldots,\sigma_{\alpha}(z_i),\ldots,z_n)$ is analytic at $z_i=\alpha$.
\end{itemize}

The recursion \eqref{eq:genrec1} satisfied locally by $H^{[a]}_{g,n}(x_1,\ldots,x_n)$ is satisfied globally by $F^{[a]}_{g,n}(z_1,\ldots,z_n)$.  Recall that $S'=\{2,\ldots,n\}$ and $S=\{1,\ldots,n\}$.  For $2g-2+n>1$,
\begin{equation}\label{eq:genrec2}\begin{split}
\left( 2g-2+n+\frac{1}{a}\sum_{i=1}^n(1-az_i^a)x_i\frac{\partial}{\partial x_i}\right)&F^{[a]}_{g,n}(z_{S}))
=\\
\sum_{i< j}&\frac{\left(\frac{z_j}{1-az_i^a}x_i\frac{\partial}{\partial x_i}-\frac{z_i}{1-az_j^a}x_j\frac{\partial}{\partial x_j}\right)}{z_i-z_j}\Big(F^{[a]}_{g,n-1}(z_{S \setminus \{j\}})+F^{[a]}_{g,n-1}(z_{S \setminus \{i\}})\Big)\\
+\frac{1}{2}\sum_{i=1}^nx(t_1)x(t_2)&\frac{\partial^2}{\partial x(t_1)\partial x(t_2)}F^{[a]}_{g-1,n+1}(t_1,t_2,z_{S \setminus \{i\}})\Big|_{t_1=t_2=z_i}\\
&+ \frac{1}{2}\sum_{i=1}^n \mathop{\sum^{\mathrm{stable}}_{g_1+g_2=g}}_{I\sqcup J=S \setminus \{i\}} x_i\frac{\partial}{\partial x_i} F^{[a]}_{g_1,|I|+1}(z_i,z_I)x_i\frac{\partial }{\partial x_i}F^{[a]}_{g_2,|J|+1}(z_i,z_J)
\end{split}
\end{equation}

Take the principal part of \eqref{eq:genrec2} at $z_1=\alpha$ and take the invariant part under the involution $\hat{z}_1:=\sigma_i(z_1)$.    

First we have
\[
\big[F^{[a]}_{g,n}(z_1,z_I)+F^{[a]}_{g,n}(\hat{z}_1,z_I)\big]_{\alpha}=0=\Big[x_i\frac{\partial}{\partial x_i}F^{[a]}_{g,n}(z_1,z_I)+x_i\frac{\partial}{\partial x_i}F^{[a]}_{g,n}(\hat{z}_1,z_I)\Big]_{\alpha}=0,\quad i=1,\ldots,n
\]
for any $I\subset S'$ since $F^{[a]}_{g,n}(z_1,z_I)+F^{[a]}_{g,n}(\hat{z}_1,z_I)$ and $x_i\frac{\partial}{\partial x_i}F^{[a]}_{g,n}(z_1,z_I)+x_i\frac{\partial}{\partial x_i}F^{[a]}_{g,n}(\hat{z}_1,z_I)$ are analytic at $z_1=\alpha$.  This annihilates the factor $(2g-2+n)$ in the first line of \eqref{eq:genrec2}, and all summands not involving $z_1$ in all four lines of \eqref{eq:genrec2}. 

The principal part of the terms in \eqref{eq:genrec2} involving $1/(z_1-z_j)$ are calculated as follows.  For any $j\neq 1$, put $\cf_j(z_1)=x_1\frac{\partial}{\partial x_1}F^{[a]}_{g,n-1}(z_1,z_{S' \setminus \{j\}})$.  Then $\cf_j(z_j)=x_j\frac{\partial}{\partial x_j}F^{[a]}_{g,n-1}(z_{S \setminus \{1\}})$ and
\begin{align*}
 \left[\frac{\frac{z_j}{1-az_1^a}\cf_j(z_1)-\frac{z_1}{1-az_j^a}\cf_j(z_j)}{z_1-z_j}\right]_{\alpha}&=
 \left[\frac{\frac{z_1}{1-az_1^a}\cf_j(z_1)-\frac{z_j}{1-az_j^a}\cf_j(z_j)}{z_1-z_j}-\frac{\cf_j(z_1)}{1-az_1^a}-\frac{\cf_j(z_j)}{1-az_j^a}\right]_{\alpha}\\
&=\left[\frac{\frac{z_1}{1-az_1^a}\cf_j(z_1)}{z_1-z_j}\right]_{\alpha}-\left[\frac{\cf_j(z_1)}{1-az_1^a}\right]_{\alpha}\\
&=\left[\frac{\frac{z_1}{1-az_1^a}\cf_j(z_1)}{z_1-z_j}\right]_{\alpha}+c_j
\end{align*}
where we have used \eqref{eq:princ} and the fact that $\frac{\cf_j(z_j)}{1-az_j^a}$ is independent of $z_1$ and hence annihilated by taking principal parts.  Note that $c_j:=-\left[\frac{\cf_j(z_1)}{1-az_1^a}\right]_{\alpha}$ is independent of $z_j$. 
Thus the invariant part of the principal part of \eqref{eq:genrec2} becomes:

\begin{equation}\label{eq:invt}
\Big[\left(\hat{z}_1^a-z_1^a\right)x_1\frac{\partial}{\partial x_1}F^{[a]}_{g,n}(z_1,z_{S'})\Big]_{\alpha}=
\Bigg[\sum_{j=2}^n\frac{\frac{z_1}{1-az_1^a}x_1\frac{\partial }{\partial x_1}F^{[a]}_{g,n-1}(z_1,z_{S' \setminus \{j\}})}{z_1-z_j}
+\frac{\frac{\hat{z}_1}{1-a\hat{z}_1^a}x_1\frac{\partial }{\partial x_1}F^{[a]}_{g,n-1}(\hat{z}_1,z_{S' \setminus \{j\}})}{\hat{z}_1-z_j}\Bigg]_{\alpha}
\end{equation}
\[
+\Bigg[\frac{1}{2}x(t_1)x(t_2)\Bigg(\frac{\partial^2}{\partial x(t_1)\partial x(t_2)}F^{[a]}_{g-1,n+1}(t_1,t_2,z_{S'})\Big|_{t_1=t_2=z_1}+\frac{\partial^2}{\partial x(t_1)\partial x(t_2)}F^{[a]}_{g-1,n+1}(t_1,t_2,z_{S'})\Big|_{t_1=t_2=\hat{z}_1}\Bigg)\Bigg]_{\alpha}
\]
\[
+\frac{1}{2} \mathop{\sum^{\mathrm{stable}}_{g_1+g_2=g}}_{I\sqcup J=S'} \Big[x_1\frac{\partial}{\partial x_1} F^{[a]}_{g_1,|I|+1}(z_1,z_I)x_1\frac{\partial }{\partial x_1}F^{[a]}_{g_2,|J|+1}(z_1,z_J)+x_1\frac{\partial}{\partial x_1} F^{[a]}_{g_1,|I|+1}(\hat{z}_1,z_I)x_1\frac{\partial }{\partial x_1}F^{[a]}_{g_2,|J|+1}(\hat{z}_1,z_J)\Big]_{\alpha}
\]
 \[\hspace{-5cm}+\sum_{j=2}^nc'_j\]
where $c'_j\ \big(:=c_j(z_1)+c_j(\hat{z}_1)\big)$ is independent of $z_j$.  There is now a one-to-one correspondence between terms in \eqref{eq:invt} and terms in the Eynard--Orantin recursion \eqref{eq:EOrec} if we ignore the $c'_j$ terms which will be annihilated later.  We can simplify \eqref{eq:invt} further to:

\begin{equation}\begin{split}\label{eq:invt1}
\Big[\left(\hat{z}_1^a-z_1^a\right)x_1\frac{\partial}{\partial x_1}F^{[a]}_{g,n}(z_1,z_{S'})\Big]_{\alpha}=
\sum_{j=2}^n\Bigg[\frac{\frac{z_1}{1-az_1^a}x_1\frac{\partial }{\partial x_1}F^{[a]}_{g,n-1}(\hat{z}_1,z_{S' \setminus \{j\}})}{z_1-z_j}
+\frac{\frac{\hat{z}_1}{1-a\hat{z}_1^a}x_1\frac{\partial }{\partial x_1}F^{[a]}_{g,n-1}(z_1,z_{S' \setminus \{j\}})}{\hat{z}_1-z_j}\Bigg]_{\alpha}
\\
+\Bigg[x(t_1)x(t_2)\frac{\partial^2}{\partial x(t_1)\partial x(t_2)}F^{[a]}_{g-1,n+1}(t_1,t_2,z_{S'})\Big|_{t_1=z_1,t_2=\hat{z}_1}\Bigg]_{\alpha}\quad\quad\quad\quad
\\
+ \mathop{\sum^{\mathrm{stable}}_{g_1+g_2=g}}_{I\sqcup J=S'} \Big[x_1\frac{\partial}{\partial x_1} F^{[a]}_{g_1,|I|+1}(z_1,z_I)x_1\frac{\partial }{\partial x_1}F^{[a]}_{g_2,|J|+1}(\hat{z}_1,z_J)\Big]_{\alpha}+\sum_{j=2}^nc'_j.
\end{split}
\end{equation}

Define $\cf(z_1)=x_1\frac{\partial}{\partial x_1}F^{[a]}_{g,n-1}(z_1,z_{S \setminus \{j\}})$.  The replacement in the first line of \eqref{eq:invt} to get the first line of \eqref{eq:invt1} uses:
\begin{equation}  \label{eq:02sym}
\Bigg[\frac{\frac{z_1}{1-az_1^a}\cf(z_1)}{z_1-z_j}+\frac{\frac{\hat{z}_1}{1-a\hat{z}_1^a}\cf(\hat{z}_1)}{\hat{z}_1-z_j}+\frac{\frac{z_1}{1-az_1^a}\cf(\hat{z}_1)}{z_1-z_j}+\frac{\frac{\hat{z}_1}{1-a\hat{z}_1^a}\cf(z_1)}{\hat{z}_1-z_j}\Bigg]_{\alpha}
=0
\end{equation}
which is true since \eqref{eq:02sym} factorises into
\begin{align*}
\Bigg[\Big(\cf(z_1)+\cf(\hat{z}_1)\Big)\left(\frac{\frac{z_1}{1-az_1^a}}{\hat{z}_1-z_j}+\frac{\frac{\hat{z}_1}{1-a\hat{z}_1^a}}{z_1-z_j}\right)\Bigg]_{\alpha}&=
\Bigg[\Big(\cf(z_1)+\cf(\hat{z}_1)\Big)\left(\frac{dz_1}{\hat{z}_1-z_j}+\frac{d\hat{z}_1}{z_1-z_j}\right)\frac{x_1}{dx_1}\Bigg]_{\alpha}
\\
&=\Bigg[\Big(\cf(z_1)+\cf(\hat{z}_1)\Big)\frac{x_1}{x_1-x_j}\Bigg]_{\alpha}\\
&=\Bigg[\frac{x_1(\cf(z_1)+\cf(\hat{z}_1))-x_j(\cf(z_j)+\cf(\hat{z}_j))}{x_1-x_j}\Bigg]_{\alpha}=0
\end{align*}
where we have used $\displaystyle\frac{dx_1}{x_1}=\frac{dz_1(1-az_1^a)}{z_1}=\frac{d\hat{z}_1(1-a\hat{z}_1^a)}{\hat{z}_1}$ in the first line and the second line uses Lemma~\ref{th:02inv} together with the fact that $(\cf(z_1)+\cf(\hat{z}_1)$ is analytic at $z_1=\alpha$.  The final expression does not have a pole at $z_1=z_j$ and vanishes since it is an analytic function.

Define the symmetric function of two variables
\[\cf_3(t_1,t_2)=x(t_1)x(t_2)\frac{\partial^2}{\partial x(t_1)\partial x(t_2)}F^{[a]}_{g-1,n+1}(t_1,t_2,z_S).\]
so $\cf_3(t_1,t_2)+\cf_3(\hat{t}_1,t_2)$ is analytic at $t_1=\alpha$ and $\cf_3(t_1,t_2)+\cf_3(t_1,\hat{t}_2)$ is analytic at $t_2=\alpha$.  Thus $\cf_3(t_1,t_2)+\cf_3(\hat{t}_1,t_2)+\cf_3(t_1,\hat{t}_2)+\cf_3(\hat{t}_1,\hat{t}_2)$ is analytic at $t_1=\alpha$ and $t_2=\alpha$ so
 \[\Big[\cf_3(z_1,z_1)+\cf_3(\hat{z}_1,\hat{z}_1)+\cf_3(z_1,\hat{z}_1)+\cf_3(\hat{z}_1,z_1)\Big]_{\alpha}=0\]
which gives the second line of \eqref{eq:invt1}.  For any choice of $I\subset S$, and $J=S-I$, put 
\[\cf_1(z_1)=x_1\frac{\partial}{\partial x_1} F^{[a]}_{g_1,|I|+1}(z_1,z_I),\ \cf_2(z_1)=x_1\frac{\partial }{\partial x_1}F^{[a]}_{g_2,|J|+1}(\hat{z}_1,z_J)\] 
Then $\cf_1(z_1)+\cf_1(\hat{z}_1)$ and $\cf_2(z_1)+\cf_2(\hat{z}_1)$ are analytic at $z_1=\alpha$ 
and so is their product.  Hence
\[\Big[\cf_1(z_1)\cf_2(z_1)+\cf_1(\hat{z}_1)\cf_2(\hat{z}_1)+\cf_1(z_1)\cf_2(\hat{z}_1)+\cf_1(\hat{z}_1)\cf_2(z_1)\Big]_{\alpha}=0\]
which gives the third line of \eqref{eq:invt1}.\\
\\

Note that 
\[d_{z_2} \cdots d_{z_n}x_1\frac{\partial}{\partial x_1}F^{[a]}_{g,n}(z_1,\ldots,z_n)=\Omega^{[a]}_{g,n}(z_1,\ldots,z_n)\frac{x_1}{dx_1}.\]
so act on \eqref{eq:invt1} by $d_{z_2} \cdots d_{z_n}$.  (The $d_{z_1}$ derivatives are already present.)
\begin{equation}\begin{split}\label{eq:invt2}
\Big[\left(\hat{z}_1^a-z_1^a\right)\frac{x_1}{dx_1}\Omega^{[a]}_{g,n}(z_1,z_{S'})\Big]_{\alpha}=
\sum_{j=2}^n\Bigg[\frac{x_1}{dx_1}\Bigg(\frac{\frac{z_1}{1-az_1^a}\Omega^{[a]}_{g,n-1}(\hat{z}_1,z_{S' \setminus \{j\}})dz_j}{(z_1-z_j)^2}
+\frac{\frac{\hat{z}_1}{1-a\hat{z}_1^a}\Omega^{[a]}_{g,n-1}(z_1,z_{S' \setminus \{j\}})dz_j}{(\hat{z}_1-z_j)^2}\Bigg)\Bigg]_{\alpha}
\\
+\Bigg[\frac{x_1^2}{dx_1^2}\Omega^{[a]}_{g-1,n+1}(z_1,\hat{z}_1,z_{S'})\Bigg]_{\alpha}
+ \mathop{\sum^{\mathrm{stable}}_{g_1+g_2=g}}_{I\sqcup J=S'} \Bigg[\frac{x_1^2}{dx_1^2}\Omega^{[a]}_{g_1,|I|+1}(z_1,z_I)\Omega^{[a]}_{g_2,|J|+1}(\hat{z}_1,z_J)\Bigg]_{\alpha}.
\end{split}
\end{equation}
Note that $\displaystyle d_{z_2} \cdots d_{z_n}\sum_{j=2}^nc'_j=0$ since $d_{z_j}c'_j=0$.

Take 
\[ \Bigg[\frac{dx_1}{x_1}\frac{1}{\left(\hat{z}_1^a-z_1^a\right)}\eqref{eq:invt2}\Bigg]_{\alpha}\]
i.e. multiply the principal parts in \eqref{eq:invt2} by $\frac{dx_1}{x_1}\frac{1}{\left(\hat{z}_1^a-z_1^a\right)}$ which is analytic at $z_1=\alpha$ so can pass inside principal parts by taking principal parts again.
Substitute the identiy
\[ \frac{z_1dz_j}{(1-az_1^a)(z_1-z_j)^2}=\frac{x_1}{dx_1}\frac{dz_1dz_j}{(z_1-z_j)^2}=\frac{x_1}{dx_1}\omega_{0,2}(z_1,z_j)\]
to get for $2g-2+n>1$,
\begin{align} \label{eq:invt3}
\Big[\Omega^{[a]}_{g,n}(z_1,z_{S'})\Big]_{\alpha}=
\sum_{j=2}^n\Big[\frac{1}{ \hat{z}_1^a-z_1^a }\frac{x_1}{dx_1}\left(\Omega^{[a]}_{g,n-1}(\hat{z}_1,z_{S' \setminus \{j\}})\omega_{0,2}(z_1,z_j)+
\Omega^{[a]}_{g,n-1}(z_1,z_{S' \setminus \{j\}})\omega_{0,2}(\hat{z}_1,z_j)\right)\Big]_{\alpha}
\\ \nonumber
+\Bigg[\frac{1}{ \hat{z}_1^a-z_1^a }\frac{x_1}{dx_1}\Omega^{[a]}_{g-1,n+1}(z_1,\hat{z}_1,z_{S'})\Bigg]_{\alpha} + \mathop{\sum^{\mathrm{stable}}_{g_1+g_2=g}}_{I\sqcup J=S'} \Bigg[\frac{1}{ \hat{z}_1^a-z_1^a }\frac{x_1}{dx_1}\Omega^{[a]}_{g_1,|I|+1}(z_1,z_I)\Omega^{[a]}_{g_2,|J|+1}(\hat{z}_1,z_J)\Bigg]_{\alpha}.
\end{align}
Now (\ref{eq:invt3})  agrees with the Eynard--Orantin recursion expressed in terms of its principal parts in (\ref{eq:EOprinc}) for the kernel
\[ K(z_1,z)=\frac{z}{2(\hat{z}^a-z^a)(1-az^a)}\left(  \frac{1}{z-z_1}- \frac{1}{\hat{z}-z_1}\right)\frac{dz_1}{dz}.\]
Hence $\Omega^{[a]}_{g,n}(z)$ satisfies the Eynard--Orantin recursion \eqref{eq:EOrec} as required.  
 
To complete the proof we need to show the base cases $(g,n)=(0,3)$ and $(1,1)$ agree since until now we have used Proposition~\ref{th:PDE} which requires $2g-2+n>1$.  

For  $\Omega^{[a]}_{0,3}(z_1,z_2,z_3)$ we use Proposition~\ref{th:PDE1}.  We consider only the principal parts of $\Omega^{[a]}_{0,3}(z_1,z_2,z_3)$ since it is rational and hence determined by its  principal parts.  It is clear from (\ref{eq:PDE03}) that $F^{[a]}_{0,3}$ can only have simple poles hence has principal part
\[ \Big[F^{[a]}_{0,3}\Big]_{\alpha}=\frac{\lambda}{(z_1-\alpha)(z_2-\alpha)(z_3-\alpha)}\]
for some $\lambda$ which is easily calculated using (\ref{eq:PDE03}) to be $\lambda=-\alpha^3/a$.  Then $\Omega^{[a]}_{0,3}(z_1,z_2,z_3)=d_{z_1}d_{z_2}d_{z_3}F^{[a]}_{0,3}$ agrees with the Eynard-Orantin invariant which can be given via the direct formula Theorem 4.1 in \cite{EOrInv}.
\[\Big[\omega^0_3(z_1,z_2,z_3)\Big]_{\alpha}=\res_{z=\alpha}\frac{\omega^0_2(z,z_1)\omega^0_2(z,z_2)\omega^0_2(z,z_3)x(z)}{dx(z)dy(z)}=d_{z_1}d_{z_2}d_{z_3}\frac{1}{(z_1-\alpha)(z_2-\alpha)(z_3-\alpha)(\ln{x})''(\alpha)y'(\alpha)}\]
since $-\alpha^3/a=1/(\ln{x})''(\alpha)y'(\alpha)$.
 
For  $\Omega^{[a]}_{1,1}(z_1)$ we take the invariant part of the principal part of \eqref{eq:PDE11}.  
\begin{equation}\label{eq:invt11}
\Big[\left(\hat{z}_1^a-z_1^a\right)x_1\frac{d}{d x_1}F^{[a]}_{1,1}(z_1)\Big]_{\alpha}=
\Bigg[\frac{x_1^2}{dx_1^2}\frac{dz_1~d\hat{z}_1}{(z_1-\hat{z}_1)^2}\Bigg]_{\alpha}
\end{equation}
where we have used Lemma~\ref{th:02inv} to replace the invariant part of the right hand side of \eqref{eq:PDE11}.  Hence
\[ \Bigg[\Omega_{1,1}^{[a]}(z_1)\Bigg]_{\alpha}=\Big[\frac{1}{ \hat{z}_1^a-z_1^a }\frac{x_1}{dx_1}\omega_{0,2}(z_1,\hat{z}_1)\Bigg]_{\alpha}\]
which agrees with the Eynard--Orantin recursion expressed in terms of its principal parts so $\Omega^{[a]}_{1,1}(z_1)$ satisfies the Eynard--Orantin recursion \eqref{eq:EOrec} as required.  Alternatively, we can use \cite{JPT} where all 1-point functions on the right hand side of the ELSV-type formula in Theorem~\ref{th:JPT} have been calculated.  This yields
\[ F^{[a]}_{1,1}(z_1)=\frac{a}{24}\xi_{1}^{(a)}(z)-\frac{1}{24}\xi_{0}^{(a)}(z)\]
so $\Omega^{[a]}_{1,1}(z_1)=dF^{[a]}_{1,1}(z_1)$ agrees with $\omega^1_1(z_1)$ by a direct calculation of (\ref{eq:EOrec}).
 
Hence the base cases agree and $\Omega^{[a]}_{g,n}(z_1,...,z_n)=\omega^g_n(z_1,...,z_n)$ as required. 
\end{proof}

\section{String and dilaton equations} \label{sec:apps}

The general Eynard--Orantin theory of topological recursion includes string and dilaton equations, which relate $\omega_{g,n+1}$ and $\omega_{g,n}$.

\begin{theorem}[String and dilaton equations] The Eynard--Orantin invariants satisfy the following, where the summations are over the zeros of $dx$ on the spectral curve and $\Phi$ satisfies $d\Phi = y\frac{dx}{x}$.

\begin{align} \label{eq:string1}
\sum_\alpha \mathop{\mathrm{Res}~}_{z=\alpha} y(z) \, \omega_{g,n+1}(z, z_S) &= - \sum_{k=1}^n dz_k \frac{\partial}{\partial z_k} \left[ \omega_{g,n}(z_S) \frac{x_k}{dx_k} \right] \\
\sum_{\alpha} \mathop{\mathrm{Res}~}_{z=\alpha} \Phi(z) \, \omega_{g,n+1}(z_S, z) &= (2-2g-n) \, \omega_{g,n}(z_S).
\end{align}
\end{theorem}

These are modified versions of equation (A.26) and Theorem~4.7 from~\cite{EOrInv}. The adjustment is due to our use of the exponentiated form of $x$, which effectively requires us to use $\frac{dx}{x}$ in place of $dx$.


A consequence of Theorem~\ref{th:JPT} is that orbifold Hurwitz numbers can be expressed as
\[
H^{[a]}_g(\mu_1, \ldots, \mu_n) = a^{1-g+\sum \{ \mu_i/a \}} Q^{[a]}_{g,n}(\mu_1, \ldots, \mu_n) \prod_{i=1}^n C(\mu_i),
\]
where $C(\mu) = \frac{\mu^{\lfloor \mu/a \rfloor}}{\lfloor \mu/a \rfloor !}$ and $Q^{[a]}_{g,n}$ is a quasi-polynomial modulo $a$. The string and dilaton equations for orbifold Hurwitz numbers provide a relation between the quasi-polynomials $Q^{[a]}_{g,n+1}$ and $Q^{[a]}_{g,n}$. The above equation defines the values of $Q^{[a]}_{g,n}(\mu_1, \mu_2, \ldots, \mu_n)$ for positive integers $\mu_1, \mu_2, \ldots, \mu_n$. Since $Q^{[a]}_{g,n}$ is a quasi-polynomial, we can naturally extend its domain to all integers $\mu_1, \mu_2, \ldots, \mu_n$. In particular, it makes sense to evaluate these quasi-polynomials at zero.

\begin{theorem}[String equation for orbifold Hurwitz numbers] \label{th:string}
\[
Q^{[a]}_{g,n+1}(\mu_1, \ldots, \mu_n, 0) = (\mu_1 + \cdots + \mu_n) \, Q^{[a]}_{g,n}(\mu_1, \ldots, \mu_n)
\]
\end{theorem}

\begin{theorem}[Dilaton equation for orbifold Hurwitz numbers] \label{th:dilaton}
\[
\frac{\partial Q^{[a]}_{g,n+1}}{\partial \mu_{n+1}} (\mu_1, \ldots, \mu_n, 0) = (2-2g-n) \, Q^{[a]}_{g,n}(\mu_1, \ldots, \mu_n)
\]
\end{theorem}




The quasi-polynomial behaviour of $Q^{[a]}_{g,n}$ allows us to express the Eynard--Orantin invariants in the following way. The constants $A_{k_1, \ldots, k_n}^{r_1, \ldots, r_n}$ are the coefficients of the polynomials governing $Q^{[a]}_{g,n}$ for particular modulo classes, which are Hurwitz--Hodge integrals by Theorem~\ref{th:JPT}.
\begin{align*}
\omega_{g,n} &= \sum_{\mu_1, \ldots, \mu_n = 1}^\infty a^{1-g+\sum \{ \mu_i/a \}} Q^{[a]}_{g,n}(\mu_1, \ldots, \mu_n) \prod_{i=1}^n C(\mu_i) \mu_i x_i^{\mu_i} \frac{dx_i}{x_i} \\
&= \sum_{r_1, \ldots, r_n = 1}^a a^{1-g+\sum \{ r_i/a \}} \sum_{\mu_1 \equiv r_1, \ldots, \mu_n \equiv r_n} \sum_{k_1, \ldots, k_n = 0}^\text{finite} A_{k_1, \ldots, k_n}^{r_1, \ldots, r_n} \prod_{i=1}^n C(\mu_i) \mu_i^{k_i+1} x_i^{\mu_i} \frac{dx_i}{x_i} \\
&= \sum_{r_1, \ldots, r_n = 1}^a a^{1-g+\sum \{ r_i/a \}} \sum_{k_1, \ldots, k_n = 0}^\text{finite} A_{k_1, \ldots, k_n}^{r_1, \ldots, r_n} \prod_{i=1}^n \xi_{k_i+1}^{(r_i)}(x_i) \frac{dx_i}{x_i}
\end{align*}

\begin{proof}[Proof of Theorem~\ref{th:string}]
We begin with the following residue calculation.
\[
\sum_{\alpha} \mathop{\mathrm{Res}~}_{z=\alpha} y(z) \xi_k^{(r)}(z) \frac{dx}{x} = \begin{cases} 1 & \text{for $k=1$ and $r = a$} \\ 0 & \text{otherwise} \end{cases}
\]
One can prove by induction that
\[
\xi_k^{(r)}(z) = \frac{z^r p_k(z^a)}{(1-az^a)^{2k+1}}
\]
for positive integers $k$, where $p_k$ is a polynomial of degree $k$ for $1 \leq r \leq a-1$ and of degree $k-1$ for $r = a$. It follows that the residue
\[
\sum_{\alpha} \mathop{\mathrm{Res}~}_{z=\alpha} y(z) \xi_k^{(r)}(z)  \frac{dx}{x} = - \mathop{\mathrm{Res}~}_{z=\infty} z^a \frac{1-az^a}{z} \xi_k^{(r)}(z) ~dz
\]
can be non-zero only when $k = 1$ and $r = a$. In this case, the residue can be computed explicitly.
\[
- \mathop{\mathrm{Res}~}_{z=\infty} z^a \frac{1-az^a}{z} \xi_1^{(a)}(z)~dz = - \mathop{\mathrm{Res}~}_{z=\infty} \frac{a^2z^{2a-1}}{(1-az^a)^2}~dz = \mathop{\mathrm{Res}~}_{z=0} \frac{a^2}{(z^a-a)^2}\frac{1}{z}~dz = 1
\]

Consider the left hand side of \eqref{eq:string1}.
\begin{align*}
& \sum_{\alpha} \mathop{\mathrm{Res}~}_{z=\alpha} y(z) \omega_{g,n+1}(z, z_S) \\
=& \sum_{r, r_1, \ldots, r_n = 1}^a a^{1-g+\{ r/a \} + \sum \{ r_i/a \}} \sum_{k, k_1, \ldots, k_n = 0}^\text{finite} A_{k_1, \ldots, k_n, k}^{r_1, \ldots, r_n, r} \prod_{i=1}^n \xi_{k_i+1}^{(r_i)}(x_i) \frac{dx_i}{x_i} \sum_\alpha \mathop{\mathrm{Res}~}_{z=\alpha} y(z) \xi_{k+1}^{(r)} \frac{dx}{x} \\
=& \sum_{r_1, \ldots, r_n = 1}^a a^{1-g+ \sum \{ r_i/a \}} \sum_{k_1, \ldots, k_n = 0}^\text{finite} A_{k_1, \ldots, k_n, 0}^{r_1, \ldots, r_n, a} \prod_{i=1}^n \xi_{k_i+1}^{(r_i)}(x_i) \frac{dx_i}{x_i} \\
=& \sum_{\mu_1, \ldots, \mu_n = 1}^\infty a^{1-g+\sum \{ r_i/a \}} Q^{[a]}_{g,n+1}(\mu_1, \ldots, \mu_n, 0) \prod_{i=1}^n C(\mu_i) \mu_i x_i^{\mu_i} \frac{dx_i}{x_i}
\end{align*}

For the right hand side of \eqref{eq:string1}, we use the fact that $dz_k \frac{\partial}{\partial z_k} x_k^{\mu_k} = \mu_k x_k^{\mu_k} \frac{dx_k}{x_k}$.
\begin{align*}
\sum_{k=1}^n dz_k \frac{\partial}{\partial z_k} \left[ \omega_{g,n}(z_S) \frac{x_k}{dx_k} \right] &= \sum_{k=1}^n dz_k \frac{\partial}{\partial z_k} \left[ \frac{x_k}{dx_k} \sum_{\mu_1, \ldots, \mu_n = 1}^\infty H_g^{[a]}(\mu_1, \ldots, \mu_n) \prod_{i=1}^n \mu_i x_i^{\mu_i} \frac{dx_i}{x_i} \right] \\
&= \sum_{k=1}^n \sum_{\mu_1, \ldots, \mu_n = 1}^\infty \mu_k H_g^{[a]}(\mu_1, \ldots, \mu_n) \prod_{i=1}^n \mu_i x_i^{\mu_i} \frac{dx_i}{x_i} \\
&= \sum_{\mu_1, \ldots, \mu_n = 1}^\infty (\mu_1 + \cdots + \mu_n) a^{1-g+\sum \{ \mu_i/a \}} Q_{g,n}^{[a]}(\mu_1, \ldots, \mu_n) \prod_{i=1}^n C(\mu_i) \mu_i x_i^{\mu_i} \frac{dx_i}{x_i}
\end{align*}

Now compare coefficients of $\prod x_i^{\mu_i} \frac{dx_i}{x_i}$ for both of these expressions to yield the desired result.
\end{proof}

\begin{proof}[Proof of Theorem~\ref{th:dilaton}]
We begin with the following residue calculation.
\[
\sum_{\alpha} \mathop{\mathrm{Res}~}_{z=\alpha} \Phi(z) \xi_k^{(r)}(z) \frac{dx}{x} = \begin{cases} -1 & \text{for $k=2$ and $r = a$} \\ 0 & \text{otherwise} \end{cases}
\]

The equation $d\Phi = y \frac{dx}{x}$ implies that we may write $\Phi = \frac{1}{a} z^a - \frac{1}{2} z^{2a}$. For brevity, we omit the details of the remainder of the proof, which uses the same strategy as the proof of Theorem~\ref{th:string}.
\end{proof}

Let $\widehat{Q}^{[a]}_{g,n}$ denote the polynomial which governs the quasi-polynomial $Q^{[a]}_{g,n}$ in the case that all entries are divisible by $a$. Although the string and dilaton equations are not recursive by nature, they do allow us to uniquely determine these polynomials for low genus.

\begin{corollary}
In genus 0, we have the closed formula $\widehat{Q}_{0,n}(\mu_1, \ldots, \mu_n) = \frac{1}{a} (\mu_1 + \cdots + \mu_n)^{n-3}$. In genus 1, the polynomial $\widehat{Q}_{1,n+1}$ can be effectively determined from $\widehat{Q}_{1,n}$ by the string and dilaton equations. 
\end{corollary}


\begin{proof}
The formula certainly holds for the base cases $n = 1$ and $n = 2$, which correspond to the unstable cases of Theorem~\ref{th:JPT}. Now suppose that the formula is true for some $n \geq 2$. Then the string equation and the inductive hypothesis imply that
\[
\widehat{Q}^{[a]}_{0,n+1}(\mu_1, \ldots, \mu_n, 0) = (\mu_1 + \cdots + \mu_n) \, \widehat{Q}^{[a]}_{g,n}(\mu_1, \ldots, \mu_n) = \frac{1}{a}(\mu_1 + \cdots + \mu_n)^{n-2}.
\]
It follows that
\[
\widehat{Q}^{[a]}_{0,n+1}(\mu_1, \ldots, \mu_n, \mu_{n+1}) =\frac{1}{a}(\mu_1 + \cdots + \mu_n)^{n-2} + \mu_{n+1} F(\mu_1, \ldots, \mu_n, \mu_{n+1}).
\]
Now use the fact that $\widehat{Q}_{0,n+1}$ is symmetric of degree at most $n-2$, a consequence of Theorem~\ref{th:JPT}. Suppose that it is possible to write down another symmetric polynomial of degree at most $n-2$, which has the form
\[
\frac{1}{a}(\mu_1 + \cdots + \mu_n)^{n-2} + \mu_{n+1} G(\mu_1, \ldots, \mu_n, \mu_{n+1}).
\]
Then the difference $\mu_{n+1} [F(\mu_1, \ldots, \mu_n, \mu_{n+1}) - G(\mu_1, \ldots, \mu_n, \mu_{n+1})]$ must also be symmetric of degree at most $n-2$. Symmetry implies that, since it is divisible by $\mu_{n+1}$, it must also be divisible by $\mu_1 \cdots \mu_n$. The degree condition now forces the difference to be equal to zero. In other words, the symmetry and degree condition on $\widehat{Q}^{[a]}_{0,n+1}$ uniquely determine $F$ and it follows by induction that $\widehat{Q}^{[a]}_{0,n}(\mu_1, \ldots, \mu_n) = \frac{1}{a} (\mu_1 + \cdots + \mu_n)^{n-3}$.

Now use the same argument and the fact that $\widehat{Q}^{[a]}_{1,n+1}$ is symmetric of degree at most $n+1$. This allows us to determine $\widehat{Q}^{[a]}f_{1,n+1}(\mu_1, \ldots, \mu_n, \mu_{n+1})$ up to the addition of $c \mu_1 \cdots \mu_n \mu_{n+1}$ for some constant $c$. Now invoke the dilaton equation to determine the value of $c$.
\end{proof}


\appendix

\section{Graphical interpretation of Hurwitz numbers}  \label{sec:graphrep}

Let us introduce some notation for the set of branched covers enumerated by the orbifold Hurwitz numbers.

\begin{definition}
For a positive integer $a$, let $\mathrm{Cov}^{[a]}_g(\mu_1, \mu_2, \ldots, \mu_n)$ be the set of connected genus $g$ branched covers $f: C \to \mathbb{P}^1$ such that
\begin{itemize}
\item the preimages of $\infty$ are labelled $p_1, p_2, \ldots, p_n$ and the  divisor $f^{-1}(\infty)$ is equal to $\mu_1p_1 + \mu_2p_2 + \cdots + \mu_np_n$;
\item the ramification profile over 0 is given by a partition of the form $(a, a, \ldots, a)$; and
\item the only other ramification is simple and occurs over the $m$th roots of unity.
\end{itemize}
\end{definition}

Note that the weighted count of the branched covers in $\mathrm{Cov}^{[a]}_g(\mu_1, \mu_2, \ldots, \mu_n)$ is equal to $H_{g;\mu}^{[a]} \times |\mathrm{Aut}~\mu|$. The extra factor appears since we require the branched covers to have labelled preimages of $\infty$. Thus, it is natural to define the following normalisation of the orbifold Hurwitz numbers.
\[
K_g^{[a]}(\mu) = H_{g;\mu}^{[a]} \times |\mathrm{Aut}~\mu|
\]

In this appendix, we prove Proposition~\ref{th:rec} using an interpretation of Hurwitz numbers as the weighted count of fatgraphs, which appears in the work of Okounkov and Pandharipande~\cite{OPaGrom}. One can informally think of a fatgraph as the 1-skeleton of a finite cell decomposition of a connected oriented surface, where the faces are labelled from 1 up to $n$. It is useful to consider each edge as the union of two half-edges. The orientation of the underlying surface allows us to define the permutation $\sigma_0$ on the set of half-edges that cyclically permutes the half-edges adjacent to a common vertex. The permutation $\sigma_1$ denotes the fixed point free involution that swaps two half-edges comprising the same edge. The product $\sigma_2 = \sigma_0 \sigma_1$ is the permutation that cyclically permutes the half-edges adjacent to a common face. The following precise definition generalises this notion of a fatgraph by allowing $\sigma_1$ to have fixed points, which correspond to half-edges that do not get paired to create an edge. We refer to such half-edges in the fatgraph as {\em leaves}.

\begin{definition}
A {\em fatgraph} is a triple $(X, \sigma_0, \sigma_1)$ where $X$ is a finite set, $\sigma_0: X \to X$ is a permutation, and $\sigma_1: X \to X$ is an involution. We require that the group generated by $\sigma_0$ and $\sigma_1$ acts transitively on $X$ and that the elements of $X / \langle \sigma_2 \rangle$ are labelled from 1 up to $n$.
\end{definition}

The set $X / \langle \sigma_0 \rangle$ is canonically equivalent to the set of vertices of the fatgraph. The set $X / \langle \sigma_1 \rangle $ is canonically equivalent to the set of leaves and edges of the fatgraph. Furthermore, the set $X / \langle \sigma_2 \rangle$ is canonically equivalent to the set of faces of the fatgraph. The {\em perimeter} of a face is defined to be the length of the cycle of $\sigma_2$ corresponding to the face. We consider two fatgraphs $(X, \sigma_0, \sigma_1)$ and $(Y, \tau_0, \tau_1)$ to be equivalent if there exists a bijection $\phi: X \to Y$ satisfying $\phi \circ \sigma_0 = \tau_0 \circ \phi$ and $\phi \circ \sigma_1 = \tau_1 \circ \phi$, which preserves the face labels. Thus, each fatgraph $\Gamma$ is endowed with a natural automorphism group $\mathrm{Aut}~\Gamma$.


The structure of a fatgraph allows one to thicken the underlying graph to a connected oriented surface with boundary, where the boundary components naturally correspond to the faces. In particular, a fatgraph acquires a type $(g,n)$, where $g$ denotes the genus and $n$ the number of faces. The following diagram shows two distinct fatgraphs --- the first  of type $(0,3)$ and the second of type $(1,1)$ --- whose underlying graphs are isomorphic. We use the usual convention whereby the cyclic ordering of the half-edges adjacent to a vertex is induced by the orientation of the page.


\begin{center}
\includegraphics{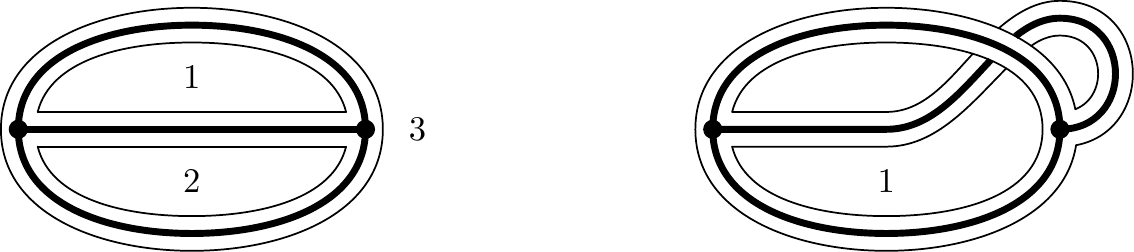}
\end{center}

\begin{definition}
For a positive integer $a$, let $\mathrm{Fat}^{[a]}_g(\mu_1, \mu_2, \ldots, \mu_n)$ be the set of edge-labelled fatgraphs of type $(g,n)$ such that
\begin{itemize}
\item there are $\frac{|\mu|}{a}$ vertices and at each of them there are $am$ adjacent half-edges that are cyclically labelled
\[
1, 2, 3, \ldots, m, 1, 2, 3, \ldots, m, \ldots, 1, 2, 3, \ldots, m;
\]
\item there are exactly $m$ edges that are labelled $1, 2, 3, \ldots, m$; and
\item the perimeters of the faces are given by the tuple $(\mu_1m, \mu_2m, \ldots, \mu_nm)$.
\end{itemize}
\end{definition}

Here, we say that an edge is labelled $k$ if its constituent half-edges are both labelled $k$. For example, the set $\mathrm{Fat}^{[2]}_0(3,1)$ consists of the following three fatgraphs, where the face labels have been omitted for clarity.

\begin{center}
\includegraphics{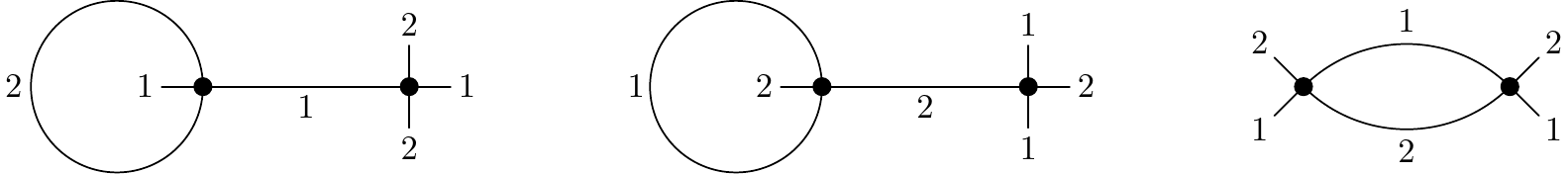}
\end{center}



\begin{proposition}
There is a one-to-one correspondence between $\mathrm{Cov}^{ [a]}_g(\mu_1,\ldots,\mu_n)$ and $\mathrm{Fat}^{ [a]}_g(\mu_1,\ldots,\mu_n)$ that preserves automorphism groups. Consequently, the normalised orbifold Hurwitz number $K_g^{[a]}(\mu)$ is the weighted count of the fatgraphs in $\mathrm{Fat}^{ [a]}_g(\mu_1,\ldots,\mu_n)$.
\end{proposition}

\begin{proof}
Let $\Gamma_m$ denote the fatgraph with one vertex obtained by connecting $0$ to the $m$th roots of unity in $\mathbb{P}^1$ by half-edges, as shown in the diagram below. The one-to-one correspondence between $\mathrm{Cov}^{ [a]}_g(\mu_1,\mu_2,\ldots,\mu_n)$ and $\mathrm{Fat}^{ [a]}_g(\mu_1,\mu_2,\ldots,\mu_n)$ is given by $f \mapsto f^{-1}(\Gamma_m)$.

\begin{center}
\includegraphics{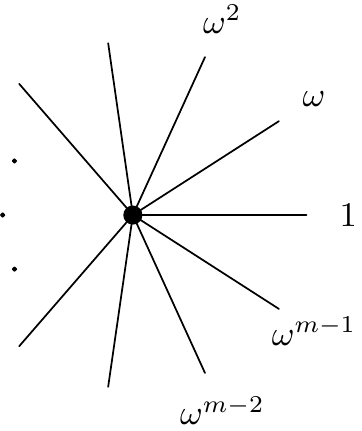}
\end{center}

The vertices correspond to the preimages of 0 and the faces to the preimages of $\infty$. Preimages of the segment connecting 0 to $\omega^k$ correspond to half-edges labelled $k$, so that the edges correspond to the points where branching occurs. The conditions for a branched cover to be in $\mathrm{Cov}^{ [a]}_g(\mu_1,\mu_2,\ldots,\mu_n)$ translate into the conditions for a fatgraph to be in $\mathrm{Fat}^{ [a]}_g(\mu_1,\mu_2,\ldots,\mu_n)$.

An isomorphism of fatgraphs is equivalent to an orientation-preserving homeomorphism of the underlying surface that maps vertices, edges, and faces to vertices, edges, and faces, while preserving all labels. It follows that the notion of equivalence and automorphism of branched covers descends to the notion of equivalence and automorphism of fatgraphs.
\end{proof}

Continuing the previous example, we note that the three fatgraphs in $\mathrm{Fat}_0^{[2]}(3,1)$ have trivial automorphism groups. Therefore, we have calculated the orbifold Hurwitz number $H_{0;(3,1)}^{[2]} = K_0^{[2]}(3,1) = 3$.

The cut-and-join recursion for orbifold Hurwitz numbers can be stated in terms of the normalisation $K_g^{[a]}(\mu)$.
\begin{align*}
K_g^{[a]}(\mu_S) &= \sum_{i < j} (\mu_i + \mu_j) K_g^{[a]}(\mu_{S \setminus \{i,j\}}, \mu_i+\mu_j) \\
&+ \sum_{i=1}^n \sum_{\alpha+\beta=\mu_i} \frac{\alpha \beta}{2} \left[ K_{g-1}^{[a]}(\mu_{S \setminus \{i\}},\alpha,\beta) + \sum_{g_1+g_2=g} \sum_{I \sqcup J = S \setminus \{i\}}\frac{(m-1)!}{m_1! m_2!} K_{g_1}^{[a]}(\mu_I, \alpha) K_{g_2}^{[a]}(\mu_J, \beta) \right],
\end{align*}
Here, we use the notation $m_1 = 2g_1 - 1 + |I| + \frac{|\mu_I|+\alpha}{a}$ and $m_2 = 2g_2 - 1 + |J| + \frac{|\mu_J|+\beta}{a}$. The conditions $g_1+g_2=g$, $I \sqcup J = S \setminus \{i\}$, and $\alpha+\beta = \mu_i$ imply that $m_1+m_2=m-1$.

\begin{proof}[Proof of Proposition~\ref{th:rec}]
Recall that $K_g^{[a]}(\mu_S)$ is the weighted count of the fatgraphs in $\mathrm{Fat}_g^{ [a]}(\mu_S)$. Choose a fatgraph in $\mathrm{Fat}_g^{ [a]}(\mu_S)$ and remove the half-edges and the edge labelled $m$ from it. Then one of the following three cases must arise.

\begin{itemize}
\item {\em The edge labelled $m$ is adjacent to two distinct faces labelled $i$ and $j$.} \\
The removal of the half-edges and the edge labelled $m$ leaves a fatgraph in $\mathrm{Fat}^{[a]}_{g}(\mu_{S \setminus \{i,j\}}, \mu_i+\mu_j)$.

\begin{center}
\includegraphics{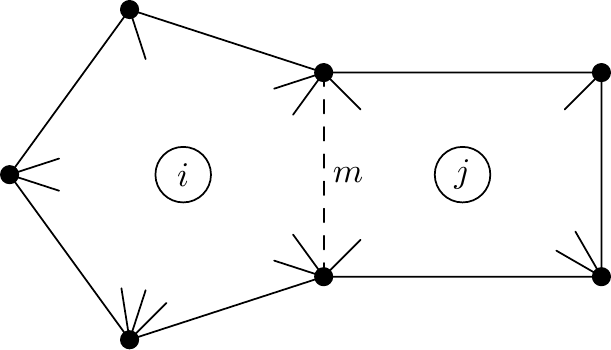}
\end{center}

\noindent Conversely, there are $\mu_i + \mu_j$ ways to reconstruct a fatgraph in $\mathrm{Fat}_g^{ [a]}(\mu_S)$ from a fatgraph in $\mathrm{Fat}^{[a]}_{g}(\mu_{S \setminus \{i,j\}}, \mu_i+\mu_j)$  by adding half-edges and one edge labelled $m$.

\item {\em The edge labelled $m$ is adjacent to the face labelled $i$ on both sides and its removal leaves a fatgraph.} \\
The removal of the half-edges and the edge labelled $m$ leaves a fatgraph in $\mathrm{Fat}^{[a]}_{g-1}(\mu_{S \setminus \{i\}}, \alpha, \beta)$, where $\alpha + \beta = \mu_i$.

\begin{center}
\includegraphics{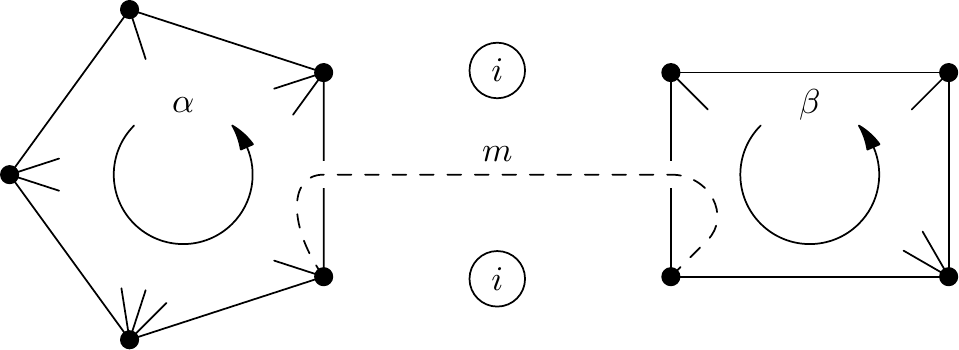}
\end{center}

\noindent Conversely, there are $\frac{1}{2} \times \alpha \times \beta$ ways to reconstruct a fatgraph in $\mathrm{Fat}_g^{ [a]}(\mu_S)$ from a fatgraph in $\mathrm{Fat}^{[a]}_{g-1}(\mu_{S \setminus \{i\}}, \alpha, \beta)$ where $\alpha + \beta = \mu_i$ by adding half-edges and one edge labelled $m$. The factor $\alpha \times \beta$ accounts for the possible locations of the ends of the edge labelled $m$. The factor $\frac{1}{2}$ adjusts for the overcounting due to the symmetry in $\alpha$ and $\beta$.

\item {\em The edge labelled $m$ is adjacent to the face labelled $i$ on both sides and its removal leaves the disjoint union of two fatgraphs.} \\
The removal of the half-edges and the edge labelled $m$ leaves the disjoint union of two fatgraphs $\Gamma_1$ and $\Gamma_2$. Remove from $\Gamma_1$ all leaves whose label does not appear on an edge of $\Gamma_1$ and replace all labels with the numbers from 1 up to $m_1$, preserving the order. Similarly, remove from $\Gamma_2$ all leaves whose label does not appear on an edge of $\Gamma_2$ and replace all labels with the numbers from 1 up to $m_2$, preserving the order. We are left with the disjoint union of two fatgraphs in $\mathrm{Fat}^{[a]}_{g_1}(\mu_I, \alpha)$ and $\mathrm{Fat}^{[a]}_{g_2}(\mu_J, \beta)$. We necessarily have the conditions $g_1+g_2 = g$, $I \sqcup J = S \setminus \{i\}$, and $\alpha + \beta = \mu_i$.

\begin{center}
\includegraphics{cut-and-join-2}
\end{center}

\noindent Conversely, there are $\frac{1}{2} \times \alpha \times \beta \times \frac{(m-1)!}{m_1! m_2!}$ ways to reconstruct a fatgraph in $\mathrm{Fat}_g^{ [a]}(\mu_S)$ from a pair of fatgraphs in $\mathrm{Fat}^{[a]}_{g_1}(\mu_I, \alpha)$ and $\mathrm{Fat}^{[a]}_{g_2}(\mu_J, \beta)$ by adding half-edges and one edge labelled $m$. We have assumed here that $g_1+g_2=g$, $I \sqcup J = S \setminus \{i\}$, and $\alpha + \beta = \mu_i$. The factor $\frac{(m-1)!}{m_1! m_2!}$ accounts for the distribution of the labels $\{1, 2, \ldots, m-1\}$ between the two fatgraphs. The factor $\alpha \times \beta$ accounts for the possible locations of the ends of the edge labelled $m$. The factor $\frac{1}{2}$ adjusts for the overcounting due to the symmetry in $(g_1, I, \alpha)$ and $(g_2, J, \beta)$.
\end{itemize}

To obtain all fatgraphs in $\mathrm{Fat}_g^{ [a]}(\mu)$, it is necessary to perform the reconstruction process in the first case for all possible values of $i$ and $j$; in the second case for all possible values of $i$ and $\alpha + \beta = \mu_i$; and in the third case for all possible values of $i$, $g_1+g_2=g$, $I \sqcup J = S \setminus \{i\}$ and $\alpha + \beta = \mu_i$. We obtain the cut-and-join recursion for orbifold Hurwitz numbers by summing up over all these contributions.
\end{proof}


\section{Combinatorics of exponential generating functions} \label{Comb_appendix}


\begin{definition} \label{cactus-node_def}
Let $d$ be a positive integer and $\nu$ be a partition of $d$. A {\em cactus-node tree of type $\nu$} is a connected graph $D$ such that:
\begin{itemize}
	\item There exists a collection $N$ called the {\em nodes} (or {
	\em cactus-nodes}):
	\[
	N = \left\{ \,\, g_i \,\, \,\left|\,\,  \begin{array}{l} i \in \{1,\ldots,l\} \mbox{, if $i\neq j$ then $g_i \cap g_j =\emptyset$, and}\\ \mbox{$g_i$ is a {\em directed} $\nu_i$-cycle in $D$ if $\nu_i>1$,} \\ \mbox{or a vertex of $D$ if $\nu_i=1$,}  \end{array} \right. \right\}
	\]
	\item There exists a collection of edges $B$ with $|B| = \ell(\nu)-1$ and $B\cap E(N) = \emptyset$ called the {\em branches}.
	\item If $c$ is a cycle in $D$ then $c\in N$.
	\item $|\mathrm{Edges}(D)|$ = $|\mathrm{Edges}(N)|+ |\mathrm{Edges}(B)|$
\end{itemize}
Call a node that is connected to exactly one branch a {\em leaf}.  
 \end{definition}

 \begin{example}
These are examples of cactus-node trees of type $\{1,1,1,3,4,5\}$. 
\begin{center}
\includegraphics[scale=0.7]{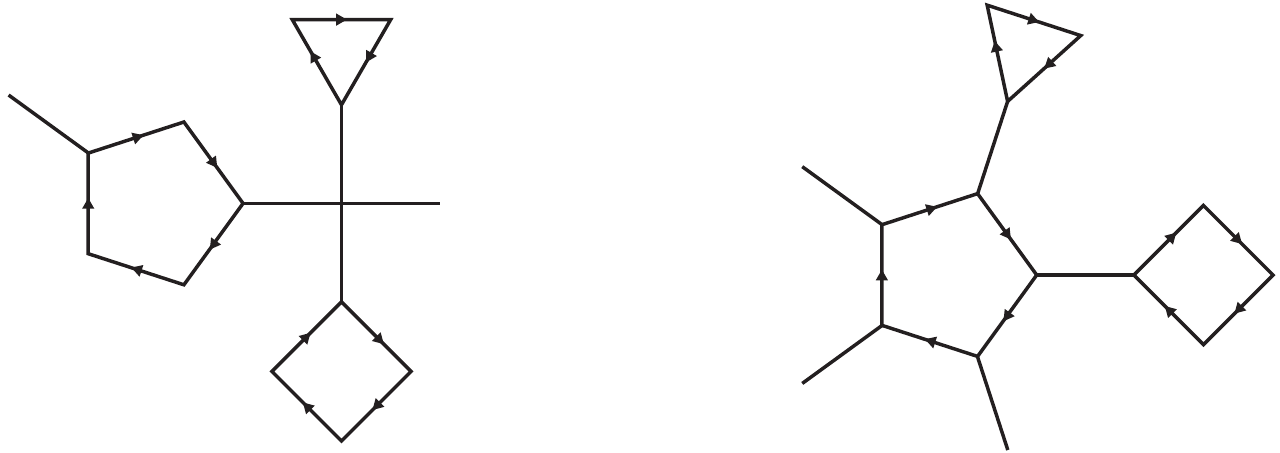}
\end{center}
\end{example}

\begin{proposition} \label{cactus-node_enum}
Let $d$ be a positive integer and $\nu$ be a partition of $d$. The number of cactus-node trees of type $\nu$  on a set of $d$ marked points is
\[
\frac{d!}{|\mathrm{Aut}~\nu|} d^{\ell(\nu)-2}.
\]
\end{proposition}
\begin{proof}
We generalise the Pr\"ufer encoding used to prove Cayley's formula for the number of labelled trees. Let $M$ be the set of collections of the form:
\[
\left\{ \,\, C_i \,\, \,\left|\,\,  \begin{array}{l} i \in \{1,\ldots,l\} \mbox{, if $i\neq j$ then $C_i \cap C_j =\emptyset$, and}\\ \mbox{$C_i$ is a {\em rooted} $\nu_i$-cycle on $\{1,\ldots,d\}$ if $\nu_i>1$,} \\ \mbox{or a marked point on $\{1,\ldots,d\}$ if $\nu_i=1$,}  \end{array} \right. \right\}
\]
We claim that there is a bijection
\[
 \left\{\begin{array}{c}\mbox{Cactus-node trees}\\ \mbox{of type $\nu$} \end{array}\right\}
\longleftrightarrow  M \times \{1,\ldots,d\}^{l-2}
\]
To see this we use ideas from the Pr\"ufer encoding:
\begin{itemize}
	\item Locate the leaf with with the largest label. 
	\item Mark the leaf where the branch is connected. 
	\item Write down the label the where branch is connected to the non-leaf component.
	\item Remove the branch connecting this to the graph.
	\item Repeat this until two leaves are left. 
	\item Remove the branch connecting the final two leaves.
	\item We are left with a collection in $M$ and a sequence in $\{1,\ldots,\}$ of length $l-2$.
\end{itemize}
This encoding can be reversed. Let $C$ be a collection in $M$, and $K$ be a  sequence in $\{1,\ldots,\}$ of length $l-2$.
\begin{itemize}
	\item Locate $b\in C$ with the largest label not in $K$. 
	\item Locate $c\in C$ that contains the label $K_1$. 
	\item Connect $b$ and $c$ with a branch at the marked points.
	\item Replace $C$ with $C\setminus b$ and $K$ with $(K_2,\ldots, K_{l-2})$.
	\item Continue until $K$ is empty.
	\item Connect the marked points of the remaining two elements of $C$.
\end{itemize}

Each $C\in M$ can be uniquely specified by
 \begin{itemize}
	\item Partitioning $S$ into sets of size determined by $\nu$. The number of ways to do this is:
 \[
 \binom{d}{\nu_1} \binom{d-\nu_1}{\nu_2}\cdots\binom{\nu_l}{\nu_l}=\frac{d!}{\nu_1!(d-\nu_1)!}\frac{(d-\nu_1)!}{\nu_2!(d-\nu_1-\nu_2)!}\cdots 1 = \frac{d!}{\nu_1! \cdots \nu_l!}. 
 \]
	\item Specifying the cycle structure and marked point of each set. For each set of size $\nu_i$ there are $\nu_i!$ possible marked cycle structures. 
	\item We must divide by $|\mathrm{Aut}~\nu|$ because these sets are unlabelled.
\end{itemize}

Hence $|M| = \dfrac{d!}{\mathrm{Aut}(\nu)}$ and the desired result follows immediately.
\end{proof}


\setlength{\parskip}{0pt}
\end{document}